\documentclass[reqno]{amsart}
\usepackage{amssymb}
\usepackage[mathscr]{eucal}
\usepackage{hyperref}
\usepackage{graphicx}
\usepackage{tikz}
\usetikzlibrary{cd,babel}

\theoremstyle{plain}
\newtheorem{prop}{Proposition}[section]
\newtheorem{thm}[prop]{Theorem}
\newtheorem{cor}[prop]{Corollary}
\newtheorem{lem}[prop]{Lemma}

\theoremstyle{definition}
\newtheorem{dfn}[prop]{Definition}
\newtheorem{rem}[prop]{Remark}
\newtheorem{rems}[prop]{Remarks}
\newtheorem{example}[prop]{Example}
\newtheorem{lab}[prop]{}

\newcommand{\isoto}{\overset{\sim}{\to}}
\newcommand{\labelto}[1]{\overset{#1}{\longrightarrow}}
\newcommand{\labelot}[1]{\overset{#1}{\longleftarrow}}

\newcommand{\A}{{\mathbb{A}}}
\newcommand{\N}{{\mathbb{N}}}
\renewcommand{\P}{{\mathbb{P}}}
\newcommand{\R}{{\mathbb{R}}}

\newcommand{\Z}{{\mathbb{Z}}}

\newcommand{\fra}{{\mathfrak{a}}}
\newcommand{\m}{{\mathfrak{m}}}

\newcommand{\scrO}{{\mathscr{O}}}

\newcommand{\sfS}{\mathsf{S}}

\DeclareMathOperator{\Hom}{Hom}

\DeclareMathOperator{\rk}{rk}
\DeclareMathOperator{\spdeg}{spdeg}
\DeclareMathOperator{\sosdeg}{sosdeg}
\DeclareMathOperator{\sosx}{sosx}
\DeclareMathOperator{\sxdeg}{sxdeg}
\DeclareMathOperator{\Spec}{Spec}
\DeclareMathOperator{\spn}{span}
\DeclareMathOperator{\tr}{tr}

\newcommand{\cone}{\mathrm{cone}}
\newcommand{\conv}{\mathrm{conv}}

\newcommand{\pr}{\mathrm{pr}}

\newcommand{\Sym}{\mathrm{Sym}}

\newcommand{\bfxi}{{\boldsymbol{\xi}}}
\newcommand{\bfzeta}{{\boldsymbol{\zeta}}}

\newcommand{\comp}{\mathbin{\scriptstyle\circ}}
\newcommand{\du}{{\scriptscriptstyle\vee}}
\renewcommand{\emptyset}{\varnothing}
\renewcommand{\setminus}{\smallsetminus}
\newcommand{\ol}{\overline}
\newcommand{\plus}{{\scriptscriptstyle+}}

\newcommand{\wt}[1]{\widetilde{#1}}
\newcommand{\all}{\forall\,}
\newcommand{\ex}{\exists\,}
\newcommand{\To}{\Rightarrow}

\renewcommand{\subset}{\subseteq}
\renewcommand{\supset}{\supseteq}

\newcommand{\idl}[1]{\langle #1\rangle}

\renewcommand{\choose}[2]{\bigl(\genfrac{}{}{0pt}{}{#1}{#2}\bigr)}
\newcommand{\bil}[2]{\langle{#1},{#2}\rangle}

\newcommand{\sa}{semialgebraic}

\begin{document}

\title
[Second-order cone representation for convex sets in the plane]
{Second-order cone representation\\for convex sets in the plane}

\author{Claus Scheiderer}
\address
  {Fachbereich Mathematik und Statistik \\
  Universit\"at Konstanz \\
  78457 Konstanz \\
  Germany}
\email
  {claus.scheiderer@uni-konstanz.de}

\begin{abstract}
Semidefinite programming (SDP) is the task of optimizing a linear
function over the common solution set of finitely many linear matrix
inequalities (LMIs). For the running time of SDP solvers, the
maximal matrix size of these LMIs is usually more critical than their
number. The semidefinite extension degree $\sxdeg(K)$ of a convex set
$K\subset\R^n$ is the smallest number $d$ such that $K$ is a linear
image of a finite intersection $S_1\cap\dots\cap S_N$, where each
$S_i$ is a spectrahedron defined by a linear matrix inequality of
size $\le d$. Thus $\sxdeg(K)$ can be seen as a measure for the
complexity of performing semidefinite programs over the set $K$. We
give several equivalent characterizations of $\sxdeg(K)$, and use
them to prove our main result: $\sxdeg(K)\le2$ holds for any closed
convex \sa\ set $K\subset\R^2$. In other words, such $K$ can be
represented using the second-order cone.
\end{abstract}

%

\date\today
\maketitle


\section*{Introduction}

Semidefinite programming (SDP) is the task of optimizing a linear
function over the solution set of a linear matrix inequality (LMI)
\begin{equation}\label{lmi}%
A_0+\sum_{i=1}^nx_iA_i\>\succeq\>0
\end{equation}
where $A_0,\dots,A_n$ are real symmetric matrices of some size,
and $A\succeq0$ means that $A$ is positive semidefinite. Under mild
conditions, semidefinite programs can be solved in polynomial time
up to any prescribed accuracy. Thanks to the enormous expressive
power of LMIs, semidefinite programming has numerous applications
from a wide range of areas. See \cite{al} for background on SDP.

Solution sets $S\subset\R^n$ of LMIs \eqref{lmi} are called
spectrahedra. So the feasible sets of SDP are spectrahedra and, more
generally, linear images of spectrahedra (aka spectrahedral shadows).
Generally, the performance of SDP solvers is strongly influenced by
the matrix size of the LMI. It is therefore desirable to express a
given feasible set by an LMI of smallest possible size. Both upper
and lower bounds for the matrix size have been studied in a number
of papers.
Here we adopt a point of view that was introduced by Averkov
\cite{av}. It is motivated by the observation that it is often
possible to represent a given convex set $K$ by the combination of
finitely many LMIs of small size~$d$. Practical experience \cite{mi}
shows that
this size $d$ is far more critical for the running time than the
number $N$ of the LMIs. Following Averkov, we define the
\emph{semidefinite extension degree} $\sxdeg(K)$ of a (convex) set
$K\subset\R^n$ as the smallest number $d$ such that $K$ is a linear
image of a finite intersection $S_1\cap\cdots\cap S_N$ of
spectrahedra that are all described by LMIs of size $\le d$. For
example, $\sxdeg(K)\le1$ if
and only if $K$ is a polyhedron, and $\sxdeg(K)\le2$ if and only if
$K$ is second-order cone representable.

Fawzi \cite{fw1} showed that the $3\times3$ psd matrix cone is not
second-order cone representable, or in other words, that
$\sxdeg(\sfS^3_\plus)=3$.
Soon after, Averkov found a general condition of combinatorial
geometric nature that is an obstruction against $\sxdeg(K)\le d$, see
\cite[Main Thm 2.1]{av} and Theorem \ref{avthm2019} below.
His proof generalizes Fawzi's techniques and uses elaborate
combinatorial techniques from Ramsey theory. As a consequence, he was
able to prove for a variety of prominent
cones (like sums of squares cones, psd matrix cones) that their
semidefinite extension degrees are not smaller than indicated by
their standard representations.
Saunderson \cite{sa} generalized Averkov's obstruction from
$\sfS_\plus^d\times\cdots\times\sfS_\plus^d$-lifts of convex sets to
$C\times\cdots\times C$-lifts, where $C$ can be an arbitrary cone
without long chains of faces.

Our main result is:

\begin{thm}\label{mainthm}%
Any closed convex \sa\ set $K\subset\R^2$
is second-order cone representable, i.e.\
we have
$\sxdeg(K)\le2$.
\end{thm}

From \cite{sch:curv} it is known that every convex \sa\ subset of
$\R^2$ is a spectrahedral shadow. So far, however, no general bounds
for the size of representing LMIs are known. To prove the main
theorem we first provide an alternative characterization of
$\sxdeg(K)$ that uses a different and more algebraic setup. Let
$K\subset\R^n$ be a convex \sa\ set, let $R$ be a real closed field
that contains the real numbers~$\R$, and let $K_R\subset R^n$ be the
base field extension of $K$ (described by the same finite system of
polynomial inequalities as~$K$). Given a point $a\in K_R$ and a
linear polynomial $f\in R[x_1,\dots,x_n]$ with $f\ge0$ on $K_R$, we
define the tensor evaluation $f^\otimes(a)$ as an element of the ring
$R\otimes R=R\otimes_\R R$. We show that $K$ is a spectrahedral
shadow if and only if $f^\otimes(a)$ is a sum of squares in
$R\otimes R$ for any choice of $R$, $f$ and~$a$ (Corollary
\ref{spshadiffpsdsos}). More precisely, $\sxdeg(K)\le d$ holds if and
only if $f^\otimes(a)$ can be written as a sum of squares of tensors
of rank $\le d$, for all $R$, $f$ and~$a$ (Theorem \ref{sxdegsosx}).
In this way, the task of proving Theorem \ref{mainthm} gets
transformed into finding a suitable
algebraic decomposition of the tangent to a plane algebraic curve at
a general point (Theorem \ref{txyidentaet}).

Although this approach appears to be highly abstract, we point out
that it is essentially constructive. Given an explicit set
$K\subset\R^2$ which is closed, convex and \sa, one can in principle
construct a second-order representation of $K$ in finitely many
steps, see Section \ref{constrasp}. The examples discussed in this
section are of much less technical nature than the general case. At
the same time, they illustrate a number of key ideas for the general
approach.

We expect that applications of this method are not confined to convex
sets in the plane:

(1) Let $K\subset\R^n$ be the closed convex hull of an algebraic
curve, or more generally, of a one-dimensional \sa\ set. From
\cite{sch:curv} it is known that $K$ is a spectrahedral shadow. We
conjecture that always $\sxdeg(K)\le\lfloor\frac n2\rfloor+1$ holds.
The bound is reached (for even $n$) by the convex hull of the
rational normal curve, see Averkov \cite[Corollary 2.3]{av}.
Note that Theorem \ref{mainthm} proves this conjecture for $n=2$.

(2)
If $K\subset\R^n$ is a compact convex body whose boundary is
smooth and has strictly positive curvature, then $K$ is known to
be a spectrahedral shadow, by results of Helton and Nie \cite{hn}.
Using the techniques developed in this paper, it can be shown that
$\sxdeg(K)=2$ holds in this case.

\begin{lab}\label{notconv}%
\emph{Notations and conventions}.
By $\sfS^d$ we denote the space of symmetric real $d\times d$
matrices, equipped with the standard inner product $\bil AB=\tr(AB)$.
We write $A\succeq B$ (resp.\ $A\succ0$) to indicate that $A-B$ is
positive semidefinite (resp.\ positive definite). The psd matrix cone
is denoted by $\sfS^d_\plus=\{A\in\sfS^d\colon A\succeq0\}$.

An $\R$-algebra is a ring $A$ together with a specified ring
homomorphism $\R\to A$. If $U\subset A$ is an $\R$-linear subspace
then $\Sigma U^2\subset A$ denotes the set of all (finite) sums of
squares of elements from $U$. Moreover $UU:=\spn(\Sigma U^2)$ is the
$\R$-linear subspace of $A$ spanned by all products $uu'$
($u,\,u'\in U$).

Algebraic varieties need neither be irreducible nor reduced. Thus an
affine $\R$-variety $X$ is just given by a finitely generated
$\R$-algebra~$A$. We write $X=\Spec(A)$ or $A=\R[X]$, and call
$A=\R[X]$ the affine coordinate ring of~$X$, as usual. Any morphism
$\phi\colon X\to Y$ of affine $\R$-varieties determines the pull-back
homomorphism $\phi^*\colon\R[Y]\to\R[X]$ between their coordinate
rings and, conversely, is determined by $\phi^*$. If $\R\subset E$ is
a field extension, the set of $E$-rational points of $X=\Spec(A)$ is
$X(E)=\Hom_\R(A,E)$ (set of homomorphisms $A\to E$ of $\R$-algebras).

For a set $K$ in $\R^n$, the convex hull of $K$ is $\conv(K)$, the
conic hull of $K$ is $\cone(K)=\{0\}\cup\bigl\{\sum_{i=1}^ra_ix_i
\colon r\ge1$, $x_i\in K$, $a_i\ge0\bigr\}$.
Throughout the paper, $P_K:=\{f\in\R[x_1,\dots,x_n]$: $f|_K\ge0$,
$\deg(f)\le1\}$ denotes the convex cone of all affine-linear
functions that are non-negative on~$K$.
\end{lab}

\textbf{Acknowledgements}.
This work was started on and inspired by the Oberwolfach meeting
\emph{Mixed-integer nonlinear optimization} in June 2019. I am
grateful to the organizers for inviting me, and to Gennadiy Averkov
for stimulating discussions and valuable suggestions. I would also
like to thank the referees for very helpful remarks. This work was
supported by DFG grant SCHE281/10-2, and also partially supported by
the EU Horizon 2020 program, grant agreement 813211 (POEMA).


\section{Semidefinite extension degree: Basic properties}

\begin{lab}\label{dfnsdeg}%
Let $n\ge1$. For any \sa\ set $S\subset\R^n$ let $\spdeg(S)$ be the
\emph{spectrahedral degree} of $S$, defined as follows. If $S$ is an
affine subspace of $\R^n$ put $\spdeg(S)=0$. Otherwise let
$\spdeg(S)$ be the smallest $d\ge1$ such that there are $m\ge1$ and an
affine-linear map $\varphi\colon\R^n\to(\sfS^d)^m=\sfS^d\times\cdots
\times\sfS^d$ with $S=\varphi^{-1}(\sfS^d_\plus\times\cdots\times
\sfS^d_\plus)$. If no such $d$ exists we put $\spdeg(S)=\infty$.
\end{lab}

So $\spdeg(S)<\infty$ if and only if $S$ is a spectrahedron, in which
case $\spdeg(S)$ is the smallest $d$ such that $S$ is the common
solution set of finitely many linear matrix inequalities of size
$d\times d$.
The notion $\spdeg(S)$ plays only a transitory role here of auxiliary
nature; as far as we know, it hasn't been considered before.

\begin{lab}\label{dfnsxdeg}%
(See Averkov \cite{av})
For a subset $K\subset\R^n$ we define the \emph{semidefinite
extension degree} $\sxdeg(K)$ as
$$\sxdeg(K)\>:=\>\inf_{S,\pi}\spdeg(S),$$
with the infimum taken over all affine-linear maps $\pi\colon\R^s\to
\R^n$ (with $s\ge1$) and all spectrahedra $S\subset\R^s$ with
$K=\pi(S)$.
\end{lab}

\begin{rems}\label{sxdegbasic}%
\hfil
\smallskip

1.\
Let $K\subset\R^n$. By definition, $\sxdeg(K)$ is the smallest
$d\ge0$ for which there is a diagram $\R^n\labelot f\R^s\labelto
\varphi\sfS^d\times\cdots\times\sfS^d$ with affine-linear maps
$\varphi,\,f$, such that $K=f(\varphi^{-1}(\sfS^d_\plus\times\cdots
\times\sfS^d_\plus))$. This almost agrees with Averkov's definition
\cite[Definition 1.1]{av}, except that \cite{av} requires in
addition that the map $\varphi$ is injective. Both definitions agree
whenever $K$ does not contain an affine subspace of positive
dimension.
\smallskip

2.\
If $K$ is an affine space then $\spdeg(K)=\sxdeg(K)=0$. If $K$ is a
polyhedron (and not an affine space) then $\spdeg(K)=\sxdeg(K)=1$.
In all other cases $\spdeg(K)\ge\sxdeg(K)\ge2$. By definition, $K$ is
a spectrahedral shadow if and only if $\sxdeg(K)<\infty$.
\smallskip

3.\
Let $K\subset\R^n$. By definition, $\sxdeg(K)\le d$ means that $K$
is a linear image of a spectrahedron that can be described by
finitely many LMIs of symmetric $d\times d$-matrices.
So it means that $K$ has a representation
$$K\>=\>\Bigl\{x\in\R^n\colon\ex y\in\R^m\ \all\nu=1,\dots,r\
A^{(\nu)}+\sum_ix_iB^{(\nu)}_i+\sum_jy_jC^{(\nu)}_j\succeq0\Bigr\}$$
with all matrices real symmetric of size (at most) $d\times d$.
\end{rems}

We record some elementary properties of $\sxdeg(K)$.

\begin{lem}\label{sxdegelem}%
Let $f\colon\R^n\to\R^m$ be an affine-linear map, let $K\subset\R^n$,
$K'\subset\R^m$ be subsets. Then
\begin{itemize}
\item[(a)]
$\sxdeg f(K)\le\sxdeg(K)$,
\item[(b)]
$\sxdeg f^{-1}(K')\le\sxdeg(K')$,
\item[(c)]
$\sxdeg(K\times K')\le\max\{\sxdeg(K),\,\sxdeg(K')\}$,
\item[(d)]
(if $m=n$) $\sxdeg(K\cap K')$, $\sxdeg(K+K')\le
\max\{\sxdeg(K),\,\sxdeg(K')\}$.
\end{itemize}
\end{lem}

\begin{proof}
(a) and (c) are obvious.
For (b) let $\pi\colon\R^s\to\R^m$ be affine-linear, and let
$S\subset\R^s$ a spectrahedron with $\pi(S)=K'$. Let
$$W\>:=\>\{(u,w)\in\R^n\times\R^s\colon f(u)=\pi(w)\}$$
(fibre sum, an affine-linear space), and let $\pr_1\colon W\to\R^n$,
$\pr_2\colon W\to\R^s$ be the canonical maps. Then
$S':=\pr_2^{-1}(S)$ is a spectrahedron in $W$ with $\spdeg(S')\le
\spdeg(S)$, and $\pr_1(S')=f^{-1}(\pi(S))=f^{-1}(K')$.
(d) follows from (a)--(c).
\end{proof}

The second part of (d) is also proved in \cite[Lemma 5.5]{av}.

\begin{example}\label{sxdeg2socp}%
(See \cite{fw1}, \cite{av})
The Lorentz cone $L_n=\{(x,t)\in\R^n\times\R\colon|x|_2\le t\}$ is a
spectrahedral cone with $\spdeg(L_n)\le n+1$.
It is easy to see that $L_n$ is a linear image of a linear section of
$(L_2)^{n-1}=L_2\times\cdots\times L_2$ (see e.g.\ \cite{btn}),
and therefore $\sxdeg(L_n)\le2$.
A second-order cone program (SOCP) optimizes a linear function over a
finite intersection of affine-linear preimages of Lorentz cones.
By Lemma \ref{sxdegelem}, any such intersection has $\sxdeg\le2$, and
the same is true for linear images of such sets.
So it follows that the feasible sets of SOCP are precisely the sets
$K$ with $\sxdeg(K)\le2$.
\end{example}

\begin{prop}\label{sxconvhull}%
Let $K,\,L\subset\R^n$ be convex sets.
\begin{itemize}
\item[(a)]
$\sxdeg(\cone(K))\,\le\,\max\{2,\,\sxdeg(K)\}$.
\item[(b)]
$\sxdeg(\conv(K\cup L))\,\le\,\max\{2,\,\sxdeg(K),\,\sxdeg(L)\}$.
\end{itemize}
\end{prop}

When $K$ is an unbounded polyhedron, the cone generated by $K$ need
not be closed. Therefore occurence of the number $2$ on the right
hand sides of Proposition \ref{sxconvhull} cannot be avoided.

\begin{proof}
(a)
Assume $d=\sxdeg(K)<\infty$. Then $K$ can be written in the form
$$K\>=\>\Bigl\{x\in\R^n\colon\ex y\in\R^m\ \all\nu=1,\dots,r\
A^{(\nu)}+\sum_ix_iB_i^{(\nu)}+\sum_jy_jC_j^{(\nu)}\succeq0\Bigr\}$$
with symmetric matrices $A^{(\nu)},\,B_i^{(\nu)},\,C_j^{(\nu)}$ of
size $d\times d$ ($1\le\nu\le r$). Then $\cone(K)$ is the set of
$x\in\R^n$ such that there exist $(y,s,t)\in\R^m\times\R\times\R$
with
$$sA^{(\nu)}+\sum_ix_iB^{(\nu)}_i+\sum_jy_jC^{(\nu)}_j\>\succeq\>0$$
($\nu=1,\dots,r$) and
$$\begin{pmatrix}s&x_i\\x_i&t\end{pmatrix}\>\succeq\>0,\quad
i=1,\dots,n.$$
(This elegant argument is due to Netzer and Sinn, see
\cite[Proposition 2.1]{ns}.)

(b)
Let $\wt K$ resp.\ $\wt L$ be the conic hull of $K\times\{1\}$ resp.\
$L\times\{1\}$ in $\R^n\times\R$.
Then $\conv(K\cup L)=\{x\in\R^n\colon(x,1)\in\wt K+\wt L\}$, so
assertion (b) follows from (a) and Lemma \ref{sxdegelem}(b),\,(d).
\end{proof}

\begin{prop}\label{sxdual2}%
Let $C\subset\R^n$ be a convex cone, and let $C^*$ be its dual cone.
Then $\sxdeg(C^*)\le\sxdeg(C)$.
\end{prop}

(See Averkov \cite[p.~135]{av} for the case where $C$ is closed and
pointed.)

\begin{proof}
Let $d=\sxdeg(C)$.
We first reduce to the case where the cone $C$ is spectrahedral.
There are a linear map $f\colon\R^s\to\R^n$ and a spectrahedron
$S\subset\R^s$ such that $f(S)=C$ and $\spdeg(S)=d$.
Let $S^h\subset\R^s\times\R$ be the homogenization of $S$
(\cite[1.13]{zi}),
i.e. $S^h=\cone(S\times\{1\})+\text{rc}(S)\times\{0\}$ where
$\text{rc}(S)$ is the recession cone of~$S$. Then $S^h$ is a
spectrahedral cone
with $\spdeg(S^h)\le d$,
and $C=g(S^h)$ for the linear map $g\colon\R^s\times\R\to\R^n$,
$g(x,t)=f(x)$.
Therefore $C^*$ is the preimage of the dual cone $(S^h)^*$ under
the dual linear map, and so $\sxdeg(C^*)\le\sxdeg((S^h)^*)$ by
Lemma \ref{sxdegelem}(b).
If we have proved $\sxdeg((S^h)^*)\le\sxdeg(S^h)$, we are therefore
done.

So let $C$ be a spectrahedral cone with a representation
$C=\{x\in\R^n\colon A_j(x)\succeq0$, $j=1,\dots,m\}$ where the
$A_j(x)=\sum_{k=1}^nx_kA_{jk}$ are linear matrix pencils in $\sfS^d$.
By a standard argument we can assume that the LMIs $A_j(x)\succeq0$
are strictly feasible.
Then if $\phi\colon(\sfS^d)^m\to\R^n$ is the linear map
$$\phi(B_1,\dots,B_m)\>=\>\Bigl(\sum_{j=1}^m\bil{B_j}{A_{j1}},
\,\dots,\,\sum_{j=1}^m\bil{B_j}{A_{jn}}\Bigr),$$
we have $C^*=\phi(\sfS^d_\plus\times\cdots\times\sfS^d_\plus)$.
\end{proof}

\begin{cor}\label{sxdual3}%
If $C\subset\R^n$ is a closed convex cone then $\sxdeg(C^*)=
\sxdeg(C)$.
\qed
\end{cor}

\begin{cor}\label{sxdegkpk}%
Let $K\subset\R^n$ be a convex set, let $P_K\subset\R[x_1,\dots,x_n]$
be the cone of all polynomials $f$ with $\deg(f)\le1$ and $f|_K\ge0$.
Then
$$\sxdeg(\ol K)\>\le\>\sxdeg(P_K)\>\le\>\max\{1,\,\sxdeg(K)\}.$$
Similarly $\sxdeg(K^o)\le\max\{1,\,\sxdeg(K)\}$ where $K^o$ is the
polar of $K$.
\end{cor}

\begin{proof}
The assertion is true when $K$ is a polyhedron, so we may assume
$\sxdeg(K)\ge2$.
Since $P_K$ is identified with the dual of the cone
$\tilde K=\cone(K\times1)=\{(tx\colon t\ge0$, $x\in K\}$ in
$\R^n\times\R$, the second inequality follows from Proposition
\ref{sxconvhull}(a). The first follows from Proposition \ref{sxdual2}
(and Lemma \ref{sxdegelem}(b)) since $\ol K$ is an affine-linear
section of the dual cone $(P_K)^*$.
Similarly, $K^o$ is an affine-linear section of the cone $P_K$.
\end{proof}


\section{Equivalent characterizations of sxdeg}

Let $n\in\N$, write $x=(x_1,\dots,x_n)$ and $L=\spn(1,x_1,\dots,x_n)
\subset\R[x]$ for the space of affine-linear polynomials.

\begin{lab}\label{recall}%
Let $K\subset\R^n$ be a convex set. By definition of sxdeg, $K$ is a
spectrahedral shadow if and only if $\sxdeg(K)<\infty$. In this
section we relate the precise value of $\sxdeg(K)$ to the
characterization of spectrahedral shadows that was given in
\cite{sch:hn}: If $K$ is closed then (\cite[Theorem~3.4]{sch:hn})
$K$ is a spectrahedral shadow if and only if there exists a morphism
$\phi\colon X\to\A^n$ of affine $\R$-varieties with $\phi(X(\R))=K$
such that $\phi^*(P_K)\subset\Sigma U^2$ holds for some
finite-dimensional linear subspace $U\subset\R[X]$.
\end{lab}

\begin{lab}\label{recallexplizit}%
Since it was somewhat hidden in \cite{sch:hn}, let us recall how such
$\phi$ and $U$ can be found explicitly from a lifted LMI
representation of~$K$. Let $K\subset\R^n$ be a spectrahedral shadow,
not necessarily closed. Replacing $\R^n$ by the affine hull of $K$ we
assume that $K$ has nonempty interior. Then
$K=\pi(\varphi^{-1}(\sfS^d_\plus))$ where $\pi\colon\R^n\times\R^m\to
\R^n$, $\pi(x,y)=x$ and $\varphi\colon\R^n\times\R^m\to\sfS^d$,
$\varphi(x,y)=M_0+\sum_{i=1}^nx_iM_i+\sum_{j=1}^my_jN_j$ for suitable
matrices $M_i,\,N_j\in\sfS^d$. The LMI in this representation can be
chosen to be strictly feasible, i.e.\ we can assume that
$\varphi(u,v)\succ0$ for some pair $(u,v)\in\R^n\times\R^m$.
Let
$$X\>=\>\{(x,y,Z)\in\A^n\times\A^m\times\Sym_d\colon Z^2=
\varphi(x,y)\},$$
a closed subvariety of $\A^n\times\A^m\times\Sym_d$, and let
$\phi\colon X\to\A^n$ be defined by $\phi(x,y,Z)=x$. Clearly
$\phi(X(\R))=K$.
Given $f\in P_K$, there are (by semidefinite duality \cite{ra})
a symmetric matrix $B\succeq0$ and a real number $c\ge0$ with
$f(x)=c+\bil B{M_0}+\sum_i\bil B{M_i}\,x_i$ and with $\bil B{N_j}=0$
for $j=1,\dots,m$.
Let $V$ be a symmetric matrix with $V^2=B$, let $(x,y,Z)\in X(\R)$.
Then
$$f(x)\>=\>c+\bil B{\varphi(x,y)}\>=\>c+\bil{V^2}{Z^2}\>=\>
c+\bil{ZV}{ZV}$$
as elements of $\R[X]$. (Here we write $\bil M{M'}=\tr(M^tM')$ for
arbitrary $d\times d$ matrices $M,\,M'$.) Hence $\phi^*(f)$ is a sum
of squares of elements from the subspace
$U:=\R1+\spn(z_{ij}\colon1\le i\le j\le d)$ of $\R[X]$, where
$Z=(z_{ij})$.
\end{lab}

We are going to characterize $\sxdeg(K)$ in terms of the possible
spaces $U$ in \ref{recall}. To this end we define:

\begin{dfn}\label{dfnsosdeg}%
For $K\subset\R^n$ a convex \sa\ set, let $\sosdeg(K)$ denote the
smallest integer $d\ge0$ such that there is a morphism
$\phi\colon X\to\A^n$ of affine $\R$-varieties, together with
finitely many $\R$-linear subspaces $U_1,\dots,U_r\subset\R[X]$,
satisfying:
\begin{itemize}
\item[(1)]
$K$ is contained in the convex hull of $\phi(X(\R))$,
\item[(2)]
$\dim(U_i)\le d$ ($i=1,\dots,r$),
\item[(3)]
$\phi^*(P_K)\subset\R_\plus1+(\Sigma U_1^2)+\cdots+(\Sigma U_r^2)$
(in $\R[X]$).
\end{itemize}
If there is no such $d$ we write $\sosdeg(K)=\infty$.
\end{dfn}

The goal of this section is to prove $\sxdeg(K)=\sosdeg(K)$ whenever
$K$ is closed and convex (Theorem \ref{sxdegsosdeg} below).

\begin{prop}\label{vonszuk}%
Let $K\subset\R^n$ be convex with $\sosdeg(K)=d<\infty$. Then there
are $\phi\colon X\to\A^n$ and subspaces $U_1,\dots,U_r\subset\R[X]$
as in Definition \ref{dfnsosdeg}, such that the stronger condition
\begin{itemize}
\item[$(1')$]
$K\subset\phi(X(\R))$
\end{itemize}
holds.
\end{prop}

\begin{proof}
Let $\phi$ and the $U_i$ be as in Definition \ref{dfnsosdeg}. Then
$S:=\phi(X(\R))$ is a \sa\ set with $K\subset\conv(S)$.
Construct a morphism $\psi\colon Y\to\A^n$ as follows. Let
$Z\subset\A^{n+1}$ be the hypersurface $z_0^2+\cdots+z_n^2=1$, so
$Z(\R)$ is the unit sphere in $\R^{n+1}$. Let
$Y:=X^{n+1}\times Z=X\times\cdots\times X\times Z$, and let
$\psi\colon Y\to\A^n$ be defined by
$$\psi\bigl(x_0,\dots,x_n;\,z_0,\dots,z_n\bigr)\>=\>\sum_{i=0}^n
z_i^2\,\phi(x_i).$$
Then $K\subset\conv(S)=\psi(Y(\R))$ by Carath\'eodory's theorem.
The coordinate ring of $Y$ is $\R[Y]=\R[X]\otimes\cdots\otimes\R[X]
\otimes\R[Z]$ ($n+1$ tensor factors $\R[X]$). For $0\le i\le n$ and
$1\le j\le r$ define the subspace $V_{ij}$ of $\R[Y]$ by
$$V_{ij}\>:=\>\R1\otimes\cdots\otimes U_j\otimes\cdots\otimes\R1
\otimes\R z_i$$
with $U_j$ at position~$i$.
Then $\dim(V_{ij})=\dim(U_j)\le d$ for all $i,j$, and $\psi^*(P_K)
\subset\R_\plus1+\sum_{i,j}(\Sigma V_{ij}^2)$. Indeed, if $f\in P_K$
then for $j=1,\dots,r$ there are elements $g_{jk}\in U_j$ with
$\phi^*(f)=c+\sum_{j=1}^r\sum_kg_{jk}^2$ for some $c\in\R$ with
$c\ge0$, by~(3).
Therefore, if we evaluate the pullback $\psi^*(f)\in\R[Y]$ at a
tuple $(\bfxi;\bfzeta)=(\xi_0,\dots,\xi_n;\,\zeta_0,\dots,\zeta_n)\in
X^{n+1}\times Z$ (of geometric points), we get
$$\psi^*(f)(\bfxi;\bfzeta)\>=\>\sum_{i=0}^nf(\phi(\xi_i))\cdot
\zeta_i^2\>=\>c+\sum_{i=0}^n\sum_{j=1}^m\sum_kg_{jk}(\xi_i)^2\cdot
\zeta_i^2$$
So, as an element of $\R[Y]$, we have
$$\psi^*(f)\>=\>c+\sum_{i=0}^n\sum_{j=1}^n\sum_k\bigl(1\otimes\cdots
\otimes g_{jk}\otimes\cdots\otimes1\otimes z_i\bigr)^2$$
and the tensor that gets squared in the $(i,j,k)$-summand lies in
$V_{ij}$, for each triple $(i,j,k)$. Hence $\psi$ and the $V_{ij}$
satisfy Definition \ref{dfnsosdeg} with $(1')$ instead of~$(1)$.
\end{proof}

\begin{lem}\label{sxsoscone}%
Let $A$ be an $\R$-algebra, let $U_1,\dots,U_r\subset A$ be
linear subspaces with $\dim(U_i)\le d$ ($i=1,\dots,r$). Then
$C:=\Sigma U_1^2+\cdots+\Sigma U_r^2$ is a convex cone with
$\sxdeg(C)\le d$.
\end{lem}

\begin{proof}
$C$ is a cone in the finite-dimensional subspace $\sum_{i=1}^rU_iU_i$
of $A$. By Lemma \ref{sxdegelem}(d) it suffices to prove the claim
for $r=1$, i.e.\ for $C=\Sigma U^2$ where $\dim(U)\le d$. If
$u_1,\dots,u_d$ is a system of linear generators of $U$ then the
linear map
$$\pi\colon\sfS^d\to UU,\quad(a_{ij})\mapsto\sum_{i,j}a_{ij}u_iu_j$$
satisfies $\pi(\sfS^d_\plus)=\Sigma U^2$.
\end{proof}

\begin{lem}\label{sxlesos}%
Let $K\subset\R^n$ be convex and \sa. Then $\sxdeg(\ol K)\le
\sosdeg(K)$.
\end{lem}

\begin{proof}
Let $d=\sosdeg(K)<\infty$. We can assume to have
$\phi\colon X\to\A^n$ and $U_i\subset\R[X]$ as in Proposition
\ref{vonszuk}.
If $d=0$ then $\phi^*(P_K)\subset\R_\plus1$. This implies that $K$
is an affine subspace (and so $\sxdeg(K)=0$). Indeed, otherwise there
would exist $f\in P_K$ such that $f$ is not constant on $K$.
But $\phi^*(f)=c$ is a constant, so $f\equiv c$ on the image of
$\phi$, contradicting $K\subset\phi(X(\R))$.

Let now $d\ge1$. By Corollary \ref{sxdegkpk} it suffices to show
$\sxdeg(P_K)\le d$. The convex cone
$C:=\R_\plus+\sum_{i=1}^r(\Sigma U_i^2)$ in $\R+\sum_{i=1}^rU_iU_i
\subset\R[X]$ satisfies $\sxdeg(C)\le d$ by Lemma \ref{sxsoscone},
and $\phi^*(P_K)\subset C$ holds by assumption. On the other hand,
elements of $C$ are nonnegative on $X(\R)$. Therefore every linear
$f\in\R[x]$ with $\phi^*(f)\in C$ is nonnegative on~$K$. This shows
$P_K=(\phi^*)^{-1}(C)$, so the proof is completed by Lemma
\ref{sxdegelem}(b).
\end{proof}

\begin{rem}
In Lemma \ref{sxlesos} the inequality $\sxdeg(K)\le\sosdeg(K)$
need not hold. For example $\sosdeg(K)=1$ but $\sxdeg(K)\ge2$ if $K$
is a dense but not closed convex subset of a polyhedron.
\end{rem}

The next lemma is the analogue of Lemma \ref{sxdegelem} for the
invariant $\sosdeg$:

\begin{lem}\label{sosdegelem}%
Let $f\colon\R^m\to\R^n$ be an affine-linear map, let $K\subset\R^n$
and $L\subset\R^m$ be convex sets. Then
\begin{itemize}
\item[(a)]
$\sosdeg f(L)\le\sosdeg(L)$,
\item[(b)]
$\sosdeg f^{-1}(K)\le\sosdeg(K)$,
\item[(c)]
$\sosdeg(K\times L)\le\max\{\sosdeg(K),\,\sosdeg(L)\}$.
\end{itemize}
\end{lem}

\begin{proof}
(a) and (c) are clear.
For (b) let $\phi\colon X\to\A^n$ be a morphism of affine varieties
with $K\subset\phi(X(\R))$ and $\phi^*(P_K)\subset
\sum_{i=1}^m(\Sigma U_i^2)$ with subspaces $U_i\subset\R[X]$ of
dimension $\le d$ ($i=1,\dots,m$). In the cartesian square (fibre
product)
$$\begin{tikzcd}[ampersand replacement=\&]
Y \ar{r}{g} \ar{d}[swap]{\psi} \& X \ar{d}{\phi} \\
\A^m \ar{r}{f} \& \A^n
\end{tikzcd}$$
we have $f^{-1}(K)\subset\psi(Y(\R))$.
We can assume $f^{-1}(K)\ne\emptyset$.
Then $P_{f^{-1}(K)}=f^*(P_K)$ holds.
The subspaces $V_i:=g^*(U_i)$ of $\R[Y]$ satisfy $\dim(V_i)\le d$
($i=1,\dots,m$), and
$$\psi^*(P_{f^{-1}(K)})\>=\>\psi^*f^*(P_K)\>=\>g^*\phi^*(P_K)
\>\subset\>g^*\Bigl(\sum_i\Sigma U_i^2\Bigr)\>\subset\>
\sum_i\Sigma V_i^2,$$
whence $\sosdeg(f^{-1}(K))\le d$.
\end{proof}

\begin{lem}\label{sosdeglesxdeg}%
If $K\subset\R^n$ is convex then $\sosdeg(K)\le\sxdeg(K)$.
\end{lem}

\begin{proof}
Let $\sxdeg(K)=d<\infty$, so there are affine-linear maps
$\R^n\labelot\pi\R^s\labelto\varphi(\sfS^d)^m$ such that
$K=\pi(\varphi^{-1}(C))$ for $C=(\sfS^d_\plus)^m$. By Lemma
\ref{sosdegelem} it suffices to show $\sosdeg(\sfS^d_\plus)\le d$.

To this end consider the morphism $\phi\colon M_d\to\Sym_d$ given by
$\phi(A)=AA^t$. Let $x_{ij}=x_{ji}$ be the coordinates on $\Sym_d$
and $y_{ij}$ those on $M_d$ ($1\le i,j\le d$). The ring homomorphism
$\phi^*\colon\R[\Sym_d]\to\R[M_d]$ is given by
$\phi^*(x_{ij})=\sum_ky_{ik}y_{jk}$.
For $k=1,\dots,d$ let
$$U_k\>:=\>\spn(y_{1k},\dots,y_{dk})\>\subset\>\R[M_d].$$
Since the cone $\sfS^d_\plus$ is self-dual, the linear forms on $\sfS^d$
that are nonnegative on $\sfS^d_\plus$ are precisely the linear forms
$f_B=\sum_{i,j}b_{ij}x_{ij}$, where $B=(b_{ij})\in\sfS^d_\plus$ is an
arbitrary psd matrix.
We claim that $\phi^*(f_B)\in(\Sigma U_1^2)+\cdots+(\Sigma U_d^2)$
for every $B\in\sfS^d_\plus$. To show this it suffices to consider
$B\succeq0$ with $\rk(B)=1$,
so let $B=bb^t$ with $b\in\R^n$. Then
$$\phi^*(f_B)\>=\>\sum_{i,j}b_ib_j\phi^*(x_{ij})\>=\>
\sum_{i,j,k}b_ib_jy_{ik}y_{jk}\>=\>\sum_k\Bigl(\sum_ib_iy_{ik}\Bigr)
^2$$
which shows the claim.
\end{proof}

Combining Lemmas \ref{sxlesos} and \ref{sosdeglesxdeg}, we have
proved:

\begin{thm}\label{sxdegsosdeg}%
For every convex set $K\subset\R^n$ one has
$$\sxdeg(\ol K)\>\le\>\sosdeg(K)\>\le\>\sxdeg(K).$$
In particular, $\sxdeg(K)=\sosdeg(K)$ if $K$ is closed.
\qed
\end{thm}

\begin{rem}\label{fausirem}%
We used uniform sum of squares decompositions of elements $f\in P_K$
in algebraic varieties $X$ over $\A^n$, to define $\sosdeg(K)$, and
then to characterize $\sxdeg(K)$. Alternatively, the definition of
$\sosdeg(K)$ and the above results, can be phrased in terms of
uniform decompositions into sums of squares of \sa\ (not necessarily
continuous) functions, as was suggested by Fawzi \cite{fw2}. Both
setups are directly equivalent, since every surjective \sa\ map
between \sa\ sets has a \sa\ section.
\end{rem}

\begin{rem}\label{gptrem}%
Let $K\subset\R^n$ be a closed convex set with $\sxdeg(K)\le
d<\infty$. By the preceding remark, there exist linear spaces
$U_1,\dots,U_m$ of \sa\ functions on $\R^n$ with $\dim(U_i)=d$ for
all $i$, such that every $f\in P_K$ lies in $(\Sigma U_1^2)+\cdots+
(\Sigma U_m^2)$. Let $p_{i1},\dots,p_{id}$ be a basis of $U_i$, for
$1\le i\le m$. For $x\in K$ and $1\le i\le m$ let
$A_i(x)=\bigl(p_{ij}(x)p_{ik}(x)\bigr)_{j,k}$, a psd symmetric matrix
of rank $\le1$ and size $d\times d$. For $f\in P_K$, since
$f\in\sum_{i=1}^m(\Sigma U_i^2)$, there are symmetric matrices
$B_1(f),\dots,B_m(f)\succeq0$ of size $d\times d$ such that
$$f\>=\>\sum_{i=1}^m\sum_{j,k=1}^db_{ijk}p_{ij}p_{ik}$$
where $B_i(f)=(b_{ijk})_{j,k}$. These matrices constitute an
$(\sfS^d_\plus)^m$-factorization of $K$ in the sense of Gouveia,
Parrilo and Thomas \cite{gpt}, since
$$f(x)\>=\>\sum_{i=1}^m\bigl\langle A_i(x),\,B_i(f)\bigr\rangle$$
holds for every $x\in K$ and every $f\in P_K$.
Note that the existence of such a $(\sfS^d_\plus)^m$-factorization,
for some $m$, is essentially equivalent to $\sxdeg(K)\le d$, by a
particular case of the main result of \cite{gpt}.
\end{rem}

We use our setup to re-prove Averkov's main theorem
\cite[Theorem~2.1]{av} in a somewhat more general setting. Given a
set $S$ and an integer $k\ge1$, let $\choose Sk$ denote the set of
all $k$-element subsets of~$S$.

\begin{thm}\label{avthm2019}%
\emph{(Averkov)}
Let $K\subset\R^n$ be a closed convex \sa\ set, let $d\in\N$. Suppose
that there exist subsets $S\subset K$ of arbitrarily large finite
cardinality that have the following property:
\begin{quote}
$(*)$
For every $T\in\choose Sd$ there exists $f\in P_K$ with $f=0$ on $T$
and $f>0$ on $S\setminus T$.
\end{quote}
Then $\sxdeg(K)\ge d+1$.
\end{thm}

\begin{proof}
We copy Averkov's elegant proof \cite{av} and transfer it from the
context of slack matrices to our setup. By way of contradiction,
assume $\sxdeg(K)\le d$. By Theorem \ref{sxdegsosdeg} (and Remark
\ref{fausirem}), there are linear spaces $U_1,\dots,U_m$ of
semialgebraic functions on $K$ with $\dim(U_i)\le d$ ($i=1,\dots,m$),
such that every $f\in P_K$ can be written $f=\sum_{i=1}^mg_i$ with
$g_i\in\Sigma U_i^2$ for $i=1,\dots,m$. For every $x\in K$ and
$i=1,\dots,m$ let $\lambda_{x,i}\in U_i^\du$ (dual space of $U_i$) be
defined by $\lambda_{x,i}(g):=g(x)$ ($g\in U_i$), and for every
subset $T\subset K$ write $L_i(T):=\spn(\lambda_{x,i}\colon x\in T)
\subset U_i^\du$. From property $(*)$ we infer:
\begin{itemize}
\item[$(**)$]
\emph{For every $T\in\choose Sd$ and for every $y\in S\setminus T$
there exists $1\le i\le m$ with $\lambda_{y,i}\notin L_i(T)$.}
\end{itemize}
Indeed, let $f\in P_K$ as in $(*)$, and write $f=\sum_{i=1}^mg_i$
with $g_i\in\Sigma U_i^2$. Since $f(y)\ne0$ there is $1\le i\le m$
with $g_i(y)\ne0$. On the other hand, $g_i(x)=0$ for every $x\in T$,
and so $\lambda_{y,i}$ is not a linear combination of the
$\lambda_{x,i}$ ($x\in T$).

Let $F\colon\choose Sd\to\{0,\dots,d\}^m$ be the map defined by
$$F(T)\>:=\>\bigl(\dim L_1(T),\,\dots,\,\dim L_m(T)\bigr).$$
If $|S|$ is sufficiently large then, by Ramsey's theorem for
hypergraphs, there is a set $W\in\choose S{d+1}$ such that $F$ is
constant on $\choose Wd$, see \cite[Theorem 3.4]{av} and \cite{grs}.
As in \cite{av} (claim on p~142), one shows for any
$T,\,T'\in\choose Wd$ and $1\le i\le m$, that the subspaces $L_i(T)$
and $L_i(T')$ of $U_i^\du$ have not only the same dimension, but
that they do in fact coincide.
This implies $L_i(T)=L_i(W)$ for every $T\in\choose Wd$. But this
contradicts $(**)$, as we see by taking $T\in\choose Wd$ and
$y\in W\setminus T$.
\end{proof}


\section{Local characterization of sxdeg}

In this section we use Theorem \ref{sxdegsosdeg} to prove another
characterization of $\sxdeg(K)$ which is of local nature (Theorem
\ref{sxdegsosx}). Even though it appears to be very ``abstract'', it
will be essential for the proof of our main result, see Sections
\ref{maintech} and~\ref{mainpf}.

\begin{lab}
Let $R$ be a real closed field that contains the field $\R$ of real
numbers. If $\phi\colon X\to Y$ is a morphism of affine
$\R$-varieties then $\phi_R\colon X_R\to Y_R$ denotes the base
extension of $\phi$ by $\R\to R$. Given a \sa\ set $M\subset\R^n$,
let $M_R$ denote the base field extension of $M$ to $R$ (see
\cite[Sect.~5.1]{bcr}). This is the subset of $R^n$ that is defined
by the same finite boolean combination of polynomial inequalities
as~$M$.
\end{lab}

\begin{lab}\label{valringb}%
By $B\subset R$ we denote
the canonical valuation ring of $R$, which is the convex hull of $\R$
in $R$, i.e.\ $B=\{b\in R\colon\ex a\in\R$ $-a<b<a\}$. The maximal
ideal of $B$ is $\m_B=\{b\in R\colon-\frac1n<b<\frac1n$ for
every~$n\in\N\}$. The residue field $B/\m_B$ of $B$ is $\R$, and the
residue map $B\to\R$ will be written $b\mapsto\ol b$.

We work in the $\R$-algebra $R\otimes R:=R\otimes_{\R}R$ and its
subring $B\otimes B=B\otimes_{\R}B$. The composite ring homomorphism
$B\otimes B\to B\to\R$, $b_1\otimes b_2\mapsto\ol{b_1b_2}$ will be
denoted by $\theta\mapsto\ol\theta$.

Given $\theta\in R\otimes R$, let $\rk(\theta)$ denote the
\emph{tensor rank} of $\theta$, i.e.\ the minimal number $r\ge0$
such that $\theta$ can be written as a sum of $r$ elementary tensors
$a_i\otimes b_i$ (with $a_i,\,b_i\in R$). Clearly we have
$\rk(\theta_1+\theta_2)\le\rk(\theta_1)+\rk(\theta_2)$ and
$\rk(\theta_1\theta_2)\le\rk(\theta_1)\cdot\rk(\theta_2)$.
We sometimes refer to tensors of rank $1,2,\dots$ as monomial,
binomial etc.\ tensors.
\end{lab}

\begin{dfn}\label{dfnsosx}%
Given a tensor $\theta\in R\otimes R$ which is a sum of squares in
$R\otimes R$, we define $\sosx(\theta)$ to be the smallest $d\ge0$
such that $\theta$ has a representation $\theta=1\otimes c+
\sum_{i=1}^N\theta_i^2$ with $0\le c\in R$ and $\theta_i\in
R\otimes R$ such that $\rk(\theta_i)\le d$ for $i=1,\dots,N$. If
$\theta$ is not a sum of squares in $R\otimes R$ we put
$\sosx(\theta)=\infty$.
\end{dfn}

In particular, $\sosx(\theta)=0$ if and only if $\theta=1\otimes c$
with $0\le c\in R$. We introduced this extra case only to make
Theorem \ref{sxdegsosx} below work in the $d=0$ case as well. The
following properties of sosx are obvious:

\begin{lem}\label{sosxobv}%
Let $\theta,\,\theta_1,\,\theta_2\in R\otimes R$.
\begin{itemize}
\item[(a)]
$\sosx(\theta)\le1$ iff there are $a_i,\,b_i\ge0$ in $R$ with
$\theta=\sum_ia_i\otimes b_i$.
\item[(b)]
$\sosx(\theta_1+\theta_2)\le\max\{\sosx(\theta_1),\,
\sosx(\theta_2)\}$.
\item[(c)]
If $\sosx(\theta_1)$, $\sosx(\theta_2)\ge1$ then
$\sosx(\theta_1\theta_2)\le\sosx(\theta_1)\cdot\sosx(\theta_2)$.
\qed
\end{itemize}
\end{lem}

The following simple observation is important:

\begin{prop}\label{keyobs}%
Let $\theta\in B\otimes B$. If $\ol\theta\in\R$ is strictly positive
then $\theta$ can be written in the form
$$\theta\>=\>\sum_{i=1}^mu_i\otimes v_i$$
with $u_i,\,v_i\in B$ and $\ol u_i,\,\ol v_i>0$ for every~$i$. In
particular, $\sosx(\theta)\le1$.
\end{prop}

\begin{proof}
Let $\theta=\sum_{i=1}^na_i\otimes b_i$ with $a_i,\,b_i\in B$. Write
(uniquely) $a_i=c_i+\alpha_i$, $b_i=d_i+\beta_i$ with
$c_i,\,d_i\in\R$ and $\alpha_i,\,\beta_i\in\m_B$ ($i=1,\dots,n$).
Choose strictly positive real numbers $r,\,s$ and $r_i,\,s_i$
($i=1,\dots,n$) with $r+s+\sum_{i=1}^nr_is_i=\ol\theta=
\sum_{i=1}^nc_id_i$, which is possible since $\ol\theta=\sum_{i=1}^n
\ol{a_ib_i}>0$. Then $\theta$ is equal to
{\small
$$\sum_{i=1}^n(r_i+\alpha_i)\otimes(s_i+\beta_i)+
\Bigl(r+\sum_{i=1}^n(d_i-s_i)\alpha_i\Bigr)\otimes1+
1\otimes\Bigl(s+\sum_{i=1}^n(c_i-r_i)\beta_i\Bigr)$$}%
and this decomposition has the desired form.
\end{proof}

\begin{rem}
The subset $T:=\bigl\{\sum_{i=1}^ra_i\otimes b_i\colon r\ge0$,
$a_i,\,b_i\in B$, $\ol a_i,\,\ol b_i>0\bigr\}$ of $B\otimes B$ is a
subsemiring of $B\otimes B$. It is easy to see that $T$ is
archimedean, i.e.\ $\Z+T=B\otimes B$. Indeed, if $a,\,b\in B$, choose
$m,\,n\in\N$ with $\pm\ol a<m$ and $\pm\ol b<n$, then
$$3mn+a\otimes b\>=\>(m-a)\otimes(n-b)+(m+a)\otimes n
+m\otimes(n+b),$$
and the right hand side lies in $T$.
This gives an alternative (but less explicit) proof of Proposition
\ref{keyobs}: Every $\theta\in B\otimes B$ with $\ol\theta>0$ is
strictly positive on the entire real spectrum of $B\otimes B$,
therefore $\theta\in T$ by the archimedean Positivstellensatz (e.g.\
\cite[Theorem 5.4.4]{ma}).
\end{rem}

\begin{lab}\label{tenseval}%
Let $V$ be an affine $\R$-variety, and let $R\supset\R$ be a real
closed field. We write $R[V]:=\R[V]\otimes R=\R[V]\otimes_{\R}R$ for
the extension of the coordinate ring of $V$ from $\R$ to~$R$. Recall
that $V(R)$, the set of $R$-points of $V$, is identified with the set
of $\R$-homomorphisms $\R[V]\to R$, by associating with an $R$-point
the evaluation homomorphism at this point. Given $f\in R[V]$ and
$a\in V(R)$ we define $f^\otimes(a)$, the ``outer'' or ``tensor
evaluation'' of $f$ at~$a$, to be the image of $f$ under the ring
homomorphism
$$R[V]\>=\>\R[V]\otimes R\>\labelto{a\otimes1}\>R\otimes R.$$
For example, for affine $n$-space $V=\A^n$, for $a\in R^n$ and any
$R$-polynomial $f=\sum_\alpha c_\alpha x^\alpha\in R[x]$ (with
$x=(x_1,\dots,x_n)$ and $c_\alpha\in R$) we get
$$f^\otimes(a)\>=\>\sum_\alpha a^\alpha\otimes c_\alpha\>\in\>
R\otimes R.$$
From the definition it is clear that $(f+g)^\otimes(a)=f^\otimes(a)+
g^\otimes(a)$ and $(fg)^\otimes(a)=f^\otimes(a)\cdot g^\otimes(a)$
hold. If $V=\A^n$ and $f=c_0+\sum_ic_ix_i$ is a linear polynomial
(with $c_i\in R$) then
$$f^\otimes(sa+tb)\>=\>(s\otimes1)\cdot f^\otimes(a)+(t\otimes1)\cdot
f^\otimes(b)$$
holds for any $s,\,t\in R$ with $s+t=1$.
If $\phi\colon X\to V$ is a morphism of affine $\R$-varieties, and if
$\phi_R^*\colon R[V]\to R[X]$ denotes the pullback homomorphism over
$R$, then for $f\in R[V]$ and $b\in X(R)$ we have
$$(\phi_R^*f)^\otimes(b)\>=\>f^\otimes(\phi(b)).$$
\end{lab}

\begin{lem}\label{tensevalpsd}%
Let $R\supset\R$ be real closed, let $S\subset V(\R)$ be a \sa\ set,
let $f\in R[V]$ with $f\ge0$ on $S_R$, and let $a\in S_R$. Then
$f^\otimes(a)$ is a psd element in $R\otimes R$, i.e.\ for any two
homomorphisms $\varphi_1,\,\varphi_2\colon R\to E$ into a real closed
field $E$, the image of $f^\otimes(a)$ under $\varphi\colon
R\otimes R\to E$, $a_1\otimes a_2\mapsto\varphi_1(a_1)\varphi_2(a_2)$
is nonnegative.
\end{lem}

\begin{proof}
$\varphi(f^\otimes(a))=g^\otimes(b)$, where $g=\varphi_2(f)\in E[V]$
satisfies $g\ge0$ on $S_E$, and $b=\varphi_1(a)\in S_E$. So
$\varphi(f^\otimes(a))\ge0$.
\end{proof}

\begin{rem}\label{sgntenseval}%
The ring $R\otimes R$ is an integral domain (by \cite{bo}, V.17.2,
Corollaire),
and it is an easy exercise to show that its quotient field is real,
i.e.\ has an ordering.
Therefore, in the situation of Lemma \ref{tensevalpsd}, the element
$-f^\otimes(a)$ is not psd in $R\otimes R$, and in particular is not
a sum of squares, unless it is zero. This argument will be used in
the proof of the main theorem in \ref{laststepofproof}.
\end{rem}

Recall the notation $P_K=\{f\in\R[x_1,\dots,x_n]\colon\deg(f)\le1$,
$f|_K\ge0\}$ for $K\subset\R^n$. The main result of this section is:

\begin{thm}\label{sxdegsosx}%
Let $K\subset\R^n$ be a closed and convex \sa\ set, let $P=P_K$, and
let $d\ge0$ be an integer. Moreover let $S\subset K$ and $E\subset P$
be \sa\ subsets with $K=\conv(S)$ and $P=\cone(E)$. Then the
following are equivalent:
\begin{itemize}
\item[(i)]
$\sxdeg(K)\le d$;
\item[(ii)]
$\sosx f^\otimes(a)\le d$ holds for every real closed field
$R\supset\R$, every $f\in P_R$ and every $a\in K_R$;
\item[(iii)]
$\sosx f^\otimes(a)\le d$ holds for every real closed field
$R\supset\R$, every $f\in E_R$ and every $a\in S_R$.
\end{itemize}
\end{thm}

Obviously, condition (iii) is a weakening of (ii). It proves useful
if we want to get a bound on $\sxdeg(K)$ through an analysis of the
tensors $f^\otimes(a)$. Typically, $E$ may be the union of all
extreme rays of $P$ (assuming that $K$ has non-empty interior in
$\R^n$), and $S$ may be the set of extreme points of $K$ (in the case
when $K$ is compact).

\begin{lab}\label{d=0case}%
Let us first dispose of the case $d=0$. If $K$ is an affine space,
i.e.\ $\sxdeg(K)=0$, then every $f\in P$ is a nonnegative constant on
$K$, and so $f^\otimes(a)=1\otimes c$ with $c\ge0$ for every
$f,\,a$ as in (ii). If $K$ is not an affine space, there is $f\in E$
which is not constant on $K$,
and so for $R\supsetneq\R$ there is $a\in S_R$ with $f(a)\notin\R$.
Hence $f^\otimes(a)=f(a)\otimes1$ is not of the form $1\otimes c$, so
(iii) doesn't hold with $d=0$.
\end{lab}

\begin{lab}\label{sosxlesxdeg}%
In the rest of the proof we assume $d\ge1$. To show (i) $\To$ (ii),
let $K\subset\R^n$ be a convex semialgebraic set with $\sxdeg(K)=d$.
Moreover let $R\supset\R$ be real closed, let $f\in P_R$ and
$a\in K_R$.
By Theorem \ref{sxdegsosdeg} (and Proposition \ref{vonszuk}) there is
a morphism $\phi\colon X\to\A^n$ of affine $\R$-varieties with
$K\subset\phi(X(\R))$, together with linear subspaces $U_1,\dots,U_m$
of $\R[X]$ with $\dim(U_i)\le d$, such that $\phi^*(P)\subset
\R_\plus+\sum_{i=1}^m(\Sigma U_i^2)$ holds. By Tarski's transfer
principle, the analogue of this inclusion holds over $R$ as well. So
there exist elements $u_{ij}\in U_i\otimes R$ (for $i=1,\dots,m$ and
$j=1,\dots,d$) such that
$$\phi_R^*(f)\>=\>c+\sum_{i=1}^m\sum_{j=1}^du_{ij}^2$$
holds in $\R[X]\otimes R=R[X]$, for some $0\le c\in R$.
Moreover there exists $b\in X(R)$ with $\phi(b)=a$, and we conclude
$$f^\otimes(a)\>=\>(\phi_R^*f)^\otimes(b)\>=\>1\otimes c+
\sum_{i=1}^m\sum_{j=1}^du_{ij}^\otimes(b)^2.$$
Since $\dim(U_i)\le d$ we have $\rk(u_{ij}^\otimes(b))\le d$ for all
$i,\,j$, which proves the implication (i) $\To$~(ii).
\end{lab}

\begin{lab}
The implication (ii) $\To$ (iii) in Theorem \ref{sxdegsosx} is
trivial. To
prove the converse, assume that (iii) holds. Let $R\supset\R$ be real
closed, let $f\in P_R$ and $a\in K_R$. There are $f_1,\dots,f_r\in
E_R$ (with $r=n+1$, if we want) and $0\le t_1,\dots,t_r\in R$ with
$f=\sum_{j=1}^rt_jf_j$. So
$$f^\otimes(b)\>=\>\sum_{j=1}^s(1\otimes t_j)\cdot f_j^\otimes(b)$$
for every $b\in R^n$. On the other hand, there are $a_1,\dots,a_m\in
S_R$ (again with $m=n+1$) and $0\le s_1,\dots,s_m\in R$ with
$\sum_{i=1}^ms_i=1$ and $a=\sum_{i=1}^ms_ia_i$. Therefore
$$g^\otimes(a)\>=\>\sum_{i=1}^m(s_i\otimes1)\cdot g^\otimes(a_i)$$
for every linear polynomial $g\in R[x]$ (see \ref{tenseval}).
Altogether
$$f^\otimes(a)\>=\>\sum_{i,j}(s_i\otimes t_j)\cdot f_j^\otimes(a_i)$$
which shows $\sosx f^\otimes(a)\le d$ by assumption~(iii) and Lemma
\ref{sosxobv}.
\end{lab}

\begin{lab}\label{assumeii}%
The proof of the remaining implication (ii) $\To$ (i) in Theorem
\ref{sxdegsosx} requires several steps. For Lemmas \ref{step1} to
\ref{completeproof} below let $K\subset\R^n$ be a convex \sa\ set,
write $P=P_K$, and assume that $\sosx f^\otimes(a)\le d$ holds for
every real closed field $R\supset\R$, every $a\in K_R$ and every
$f\in P_R$ (with $d\ge1$).
\end{lab}

\begin{lem}\label{step1}%
(Assumptions as in \ref{assumeii})
Given $R\supset\R$, a point $a\in K_R$ and a linear polynomial
$f\in P_R$, there exists a morphism $\phi\colon X\to\A^n$ of affine
$\R$-varieties, together with linear subspaces $U_1,\dots,U_m\subset
\R[X]$ of dimension $\le d$, such that $a\in\phi(X(R))$ and
$$\phi_R^*(f)\>\in\>\Sigma(U_1\otimes R)^2+\cdots+
\Sigma(U_m\otimes R)^2.$$
\end{lem}

\begin{proof}
Let $a=(a_1,\dots,a_n)$.
By definition of $\sosx f^\otimes(a)$, there exist finitely many
linear $\R$-subspaces $U_i\subset R$ with $\dim(U_i)\le d$
($i=1,\dots,m$) such that $f^\otimes(a)\in\Sigma(U_1\otimes R)^2+
\cdots+\Sigma(U_m\otimes R)^2$ in $R\otimes R$. Let $A$ be the
$\R$-subalgebra of $R$ that is (finitely) generated by
$a_1,\dots,a_n\in R$ and by $U_1+\cdots+U_m\subset R$, and let
$\varphi\colon\R[x_1,\dots,x_n]\to A$ be the homomorphism of
$\R$-algebras defined by $x_i\mapsto a_i$ ($i=1,\dots,n$). Let
$X=\Spec(A)$, let $\phi=\varphi^*\colon X\to\A^n$ be the morphism
of $\R$-varieties defined by $\varphi$. The $U_i$ are $\R$-linear
subspaces of $\R[X]=A$ with $\dim(U_i)\le d$. Moreover $a$ lies in
$\phi(X(R))$, corresponding to the inclusion homomorphism
$i\colon A\subset R$.
Under the inclusion $i\otimes1\colon A\otimes R\subset R\otimes R$,
the element $\phi_R^*(f)\in A\otimes R$ is mapped to $f^\otimes(a)$.
Therefore $\phi_R^*(f)$ has a representation of the desired form.
\end{proof}

\begin{lem}\label{step2}%
(Assumptions as in \ref{assumeii})
Given $R\supset\R$ and $f\in P_R$, there is a morphism
$\phi\colon X\to\A^n$ of affine $\R$-varieties with
$K\subset\phi(X(\R))$, and there are $\R$-linear subspaces
$U_i\subset\R[X]$ with $\dim(U_i)\le d$ ($i=1,\dots,m$), such that
$$\phi_R^*(f)\>\in\>\Sigma(U_1\otimes R)^2+\cdots+
\Sigma(U_m\otimes R)^2.$$
\end{lem}

\begin{proof}
For every real closed field $R'\supset R$ and every $a\in K_{R'}$,
Lemma \ref{step1} has shown that there exists an $\R$-morphism
$\phi\colon X\to\A^n$ with $a\in\phi(X(R'))$, together with
$\R$-subspaces $U_j\subset\R[X]$ satisfying $\dim(U_j)\le d$ and
$\phi_R^*(f)\in\sum_j\Sigma(U_j\otimes R)^2$. For each such $\phi$,
the image set $\phi(X(\R))$ is a semialgebraic subset of $\R^n$.
By compactness of the constructible topology of the real spectrum
(e.g.\ \cite[7.1.12]{bcr}),
this implies that there exist finitely many $\R$-morphisms
$\phi_i\colon X_i\to\A^n$ ($i=1,\dots,N$) such that
$K\subset\bigcup_{i=1}^N
\phi_i(X_i(\R))$, and for every $i=1,\dots,N$ finitely many
$\R$-subspaces $U_{ij}\subset\R[X_i]$ ($j=1,\dots,m_i$) with
$\dim(U_{ij})\le d$, such that for each $i=1,\dots,N$ we have
\begin{equation}\label{ithsumrep}%
\phi_{iR}^*(f)\>\in\>\Sigma(U_{i1}\otimes R)^2+\cdots+
\Sigma(U_{im_i}\otimes R)^2.
\end{equation}
From $\phi_1,\dots,\phi_N$ we can fabricate a single $\phi$, as
follows. Let $X:=\coprod_{i=1}^NX_i$ (disjoint sum),
and let $V_1,\dots,V_t\subset\R[X]$ be the $\R$-subspaces
$$\{0\}\times\cdots\times U_{ij}\times\cdots\times\{0\}\>\subset\>
\R[X]\>=\>\R[X_1]\times\cdots\times\R[X_N]$$
for $1\le i\le N$ and $1\le j\le m_i$, where $U_{ij}$ stands at
position $i$ in the direct product. Then $\dim(V_\nu)\le d$ for
all~$\nu$. If $\phi\colon X\to\A^n$ denotes the morphism which
restricts to $\phi_i$ on $X_i$, we clearly have
$K\subset\phi(X(\R))$. Moreover the element $\phi_R^*(f)=
(\phi_{1R}^*(f),\dots,\phi_{NR}^*(f))\in R[X]$ lies in
$$\Sigma(V_1\otimes R)^2+\cdots+\Sigma(V_t\otimes R)^2.$$
Indeed, this is clear by writing $\phi_R^*(f)$ as
$$\bigl(\phi_{1R}^*(f),0,\dots,0\bigr)+
\bigl(0,\phi_{2R}^*(f),0,\dots,0\bigr)+\cdots+
\bigl(0,\dots,0,\phi_{NR}^*(f)\bigr)$$
and using \eqref{ithsumrep} for $i=1,\dots,N$.
\end{proof}

\begin{lem}\label{step3}%
(Assumptions as in \ref{assumeii})
There is a morphism $\phi\colon X\to\A^n$ of affine $\R$-varieties,
together with $\R$-linear subspaces $U_1,\dots,U_m\subset\R[X]$ with
$\dim(U_i)\le d$, such that $K\subset\phi(X(\R))$ and
$\phi^*(P)\subset(\Sigma U_1^2)+\cdots+(\Sigma U_m^2)$.
\end{lem}

\begin{proof}
By Lemma \ref{step2} there exists, for every $R\supset\R$ and every
$f\in P_R$, a morphism $\phi\colon X\to\A^n$ of affine $\R$-varieties
with $K\subset\phi(X(\R))$, together with $\R$-subspaces
$U_j\subset\R[X]$ with $\dim(U_j)\le d$ ($j=1,\dots,m$), such that
$(\phi_R^*)(f)\in\sum_j\Sigma(U_j\otimes R)^2$. For each such $\phi$,
the subset
$$\{g\in P\colon\phi^*(g)\in\Sigma U_1^2+\cdots+\Sigma U_m^2\}$$
of $P$ is \sa. Again using compactness of the constructible topology,
we conclude
that there exist finitely many
$\phi_i\colon X_i\to\A^n$ ($i=1,\dots,N$), each satisfying
$K\subset\phi_i(X_i(\R))$, and for each index $i$ there exist
finitely many $\R$-subspaces $U_{ij}\subset\R[X_i]$ ($j=1,\dots,m_i$)
of dimension $\dim(U_{ij})\le d$, such that the following is true:
For every $f\in P$ there exists an index $i\in\{1,\dots,N\}$ with
\begin{equation}\label{phiistarrep}%
\phi_i^*(f)\>\in\>\Sigma U_{i1}^2+\cdots+\Sigma U_{im_i}^2.
\end{equation}
Again we construct a single $\phi$ from $\phi_1,\dots,\phi_N$: Let
$X:=X_1\times_{\A^n}\cdots\times_{\A^n}X_N$ (fibre product over
$\A^n$ via the morphisms $\phi_i\colon X_i\to\A^n$), so $\R[X]$ is
the tensor product $\R[X_1]\otimes_{\R[x]}\cdots\otimes_{\R[x]}
\R[X_N]$ via the homomorphisms $\phi_i^*\colon\R[x]\to\R[X_i]$. The
natural morphism $\phi\colon X\to\A^n$ satisfies
$K\subset\phi(X(\R))$,
and for $f\in\R[x]$ we have
\begin{equation}\label{fstartensprod}%
\phi^*(f)\>=\>\phi_1^*(f)\otimes1\otimes\cdots\otimes1\>=\>
\cdots\>=\>1\otimes\cdots\otimes1\otimes\phi_N^*(f)
\end{equation}
in $\R[X]$. Let $V_1,\dots,V_t\subset\R[X]$ be the subspaces
$$\R1\otimes\cdots\otimes U_{ij}\otimes\cdots\otimes\R1$$
for $1\le i\le N$ and $1\le j\le m_i$, where $U_{ij}$ stands at
position $i$ in the tensor product. Then $\dim(V_\nu)\le d$ for
each~$\nu$. Given $f\in P$, let $1\le i\le N$ be an index with
\eqref{phiistarrep}. Then from \eqref{fstartensprod} we see that
$$\phi^*(f)\>\in\>\Sigma V_1^2+\cdots+\Sigma V_t^2.$$
Altogether this shows that $\phi^*(P)$ is contained in the right
hand cone, which proves the lemma.
\end{proof}

\begin{lab}\label{completeproof}%
\emph{Proof of (ii) $\To$ (i) in Theorem \ref{sxdegsosx}}.
Let $K\subset\R^n$ be closed convex and \sa, and assume~(ii) (see
\ref{assumeii}). Then Lemma \ref{step3} says that
$\sosdeg(K)\le d$. Combining this with Theorem \ref{sxdegsosdeg} we
conclude that $\sxdeg(K)\le d$ since $K$ is closed. This completes
the proof of Theorem \ref{sxdegsosx}.
\qed
\end{lab}

We record an obvious relaxation of Theorem \ref{sxdegsosx}:

\begin{cor}\label{spshadiffpsdsos}%
Let $K\subset\R^n$ be a closed convex \sa\ set. Then $K$ is a
spectrahedral shadow if and only if $f^\otimes(a)$ is a sum of
squares in $R\otimes R$, for every real closed field $R\supset\R$,
every $f\in(P_K)_R$ and every $a\in K_R$.
\end{cor}

\begin{proof}
For the ``if'' direction, assume that $f^\otimes(a)$ is a sum of
squares for all choices of $R,\,f$ and~$a$. Following the proof of
Theorem \ref{sxdegsosx}, (ii) $\To$ (i) (see \ref{assumeii}), one
sees that there exists a morphism $\phi\colon X\to\A^n$ together with
a linear subspace $U\subset\R[X]$ of finite dimension such that
$K\subset\phi(X(\R))$ and $\phi^*(P)\subset\Sigma U^2$. By Theorem
\ref{sxdegsosdeg}, this implies $\sxdeg(K)=\sosdeg(K)\le\dim(U)<
\infty$. The ``only if'' direction is obvious from Theorem
\ref{sxdegsosx}.
\end{proof}

Our proof of Theorem \ref{mainthm} depends on Theorem \ref{sxdegsosx}
in an essential way. The next section will provide the necessary
algebraic background.


\section{Tensor decomposition}\label{maintech}%

\begin{lab}\label{samwoteknikl}%
The setup in this section is somewhat technical. Before we go into
the details, we give an informal outline.

Let $K\subset\R^2$ be a closed convex \sa\ set, let $P=P_K$, the
cone of linear functions nonnegative on $K$. To prove
$\sxdeg(K)\le2$, we have to show (by Theorem \ref{sxdegsosx}) that
$\sosx f^\otimes(a)\le2$ for every $a\in K_R$ and $f\in P_R$, where
$R\supset\R$ is a real closed field. To describe the essential case,
fix an irreducible plane algebraic curve $C\subset\A^2$ over $\R$.
Take two arbitrary $R$-rational points $a\ne b$ on $C$, and let
$f=\tau_b$ be the equation of the tangent to $C$ at $b$ (we assume
that $b$ is a nonsingular $R$-point). When $\tau_b(a)>0$, we need to
show $\sosx\tau_b^\otimes(a)\le2$. This in turn will follow from
Theorem \ref{txyidentaet}, which is the main result of this section.
See Section \ref{mainpf} for a rigorous proof of the main result
Theorem \ref{mainthm} from this theorem.

From a (reduced) equation $F(x,y)=0$ for $C$ we get a uniform choice
for an equation $\tau_v$ of the tangent at nonsingular points $v$
of~$C$.
This gives a regular function $T\colon(u,v)\mapsto\tau_v(u)$ on
$C\times C$, i.e.\ an element $T\in\R[C]\otimes\R[C]$. If
$a,\,b\in C(R)$ are nonsingular $R$-points, the tensor evaluation
$\tau_b^\otimes(a)\in R\otimes R$ is the image of $T$ under the map
$\R[C]\otimes\R[C]\to R\otimes R$, $p\otimes q\mapsto p(a)\otimes
q(b)$. Roughly, Theorem \ref{txyidentaet} establishes a
decomposition of $T$ in (a~localization of) $\R[X]\otimes\R[X]$
where $X\to C$ is the normalization of~$C$. When read in
$R\otimes R$, this decomposition yields the desired conclusion
$\sosx\tau_b^\otimes(a)\le2$.

In this section we work with a plane curve $C$ over $\R$ and with
its normalization.
Throughout we could work over an arbitrary base field $k$ of
characteristic zero, except that this would require a slightly
different formulation of Theorem \ref{txyidentaet}.
Since we have no need for this greater generality, we stick to
$k=\R$.
\end{lab}

\begin{lab}\label{setup1}%
We now present the details. Let $C\subset\A^2$ be an irreducible
(reduced) curve over~$\R$, and let $\pi\colon X\to C$ be its
normalization. Let $\xi\in X(\R)$ be a point,
fixed for the entire discussion, and let $\eta=\pi(\xi)\in C(\R)$.
Let $X_0\subset X$ be an (affine) open neighborhood of $\xi$ that we
will shrink further according to our needs,  and write $A=\R[X_0]$.
Always consider $A\otimes A=A\otimes_\R A$ as an $A$-algebra via the
\emph{second}
embedding $i_2\colon A\to A\otimes A$, $a\mapsto1\otimes a$. So for
$f\in A$ and $\theta\in A\otimes A$, the notation $f\theta$ means
$(1\otimes f)\cdot\theta$. Let $\text{mult}\colon A\otimes A\to A$ be
the product map, let $I$ be its kernel. For $f\in A$ the element
$\delta(f):=f\otimes1-1\otimes f$ lies in~$I$.

We choose $X_0$ so small that the $A$-module $\Omega=\Omega_{A/\R}$
of K\"ahler differentials is freely generated by $ds$, for some
$s\in A$. For $f\in A$ define $\frac{df}{ds}\in A$ by
$df=\frac{df}{ds}ds$, as usual, and let inductively
$\frac{d^if}{ds^i}=\frac d{ds}\bigl(\frac{d^{i-1}f}{ds^{i-1}}\bigr)$
for $i\ge1$. The isomorphism $\Omega\isoto I/I^2$,
$df\mapsto\delta(f)+I^2$ of $A$-modules induces $A$-linear
isomorphisms $\Sym^d_A(\Omega)\to I^d/I^{d+1}$ for all $d\ge0$
(\cite{ega4} 17.12.4, 16.9.4).
Hence, for any $f\in A$, there are unique elements $p_i\in A$
($i\ge0$) such that for every $d\ge0$ the congruence
$$f\otimes1\>\equiv\>\sum_{i=0}^d\frac{p_i}{i!}\delta(s)^i\
(\text{mod }I^{d+1})$$
holds in $A\otimes A$,
and we have $p_i=\frac{d^if}{ds^i}$ for all $i\ge0$.
Hence the congruence
\begin{equation}\label{generictaylor}
\delta(f)\>\equiv\>\sum_{i=1}^d\frac1{i!}\frac{d^if}{ds^i}\,
\delta(s)^i\ (\text{mod }I^{d+1})
\end{equation}
holds in $A\otimes A$ for every $f\in A$ and every $d\ge1$.
(For $A=\R[x]$, this is just the general Taylor expansion
$f(x)=\sum_{i\ge0}\frac1{i!}f^{(i)}(y)\,(x-y)^i$ of $f\in\R[x]$. The
general case can be reduced to this one by localization.)
\end{lab}

\begin{lab}\label{setup3}%
Via $\pi$ we consider the affine coordinates $x,\,y$ of $\A^2$ as
elements of $A$. Assume that $C$ is not a line, i.e.\ that $1,x,y$
are $\R$-linearly independent in~$A$.
Let $val_\xi\colon\R(X)^*\to\Z$ be the discrete (Krull) valuation of
the function field $\R(X)$ that is centered at~$\xi$. Since $val_\xi$
has residue field $\R$, there are (unique) integers
$1\le m_\xi<n_\xi$ such that $\{0,m_\xi,n_\xi\}=\{val_\xi(f)\colon
0\ne f\in\R+\R x+\R y\}$.
Note that $\eta$ is a nonsingular point of $C$ if and only if
$m_\xi=1$, and that $n_\xi=2$ holds if and only if $\eta$ is
nonsingular and the tangent at $\eta$ is simple.
\end{lab}

\begin{lab}\label{setup2}%
Recall that $s\in A$ is such that $ds$ is a free generator of
$\Omega_{A/\R}$. We call
$$T_s\>:=T_s(x,y)\>=\>\frac{dx}{ds}\cdot\delta(y)-\frac{dy}{ds}\cdot
\delta(x)\>\in\>\spn_\R(1,x,y)\otimes A\subset A\otimes A$$
the \emph{tangent tensor} of $C$ (relative to~$s$). Changing $s$
results in multiplying $T_s(x,y)$ with a unit of $A$.
Note that $T_s(x,y)$ is $\R$-bilinear in $x$ and $y$ and satisfies
$T_s(y,x)=-T_s(x,y)$ and $T_s(x,1)=0$.
Moreover $T_s(x,y)\in I^2$ since $\frac{dx}{ds}dy=\frac{dy}{ds}dx$
in $\Omega$.

To explain the terminology, note that if $\R\to E$ is a field
extension and $b\in X_0(E)=\Hom_\R(A,E)$ is such that $\pi(b)$ is a
nonsingular $E$-point of $C$, then the image of $T_s(x,y)$ under
$$1\otimes b\colon\spn_\R(1,x,y)\otimes A\to\spn_\R(1,x,y)\otimes E$$
is an equation for the tangent to the curve $C$ at the $E$-point
$\pi(b)$ of~$C$.
The main result of this section is:
\end{lab}

\begin{thm}\label{txyidentaet}%
Let $A_\xi=\scrO_{X,\xi}$, and consider the tangent tensor
$T=T_s(x,y)$ (\ref{setup2}) as an element of $A_\xi\otimes A_\xi$.
Let $(m,n)=(m_\xi,n_\xi)$ as in \ref{setup3}. Then, for any local
uniformizer $t\in A_\xi$, there is a choice of sign $\pm$ such that
\begin{equation}\label{txydecomp}%
\pm T\>=\>(1\otimes t^{m-1})\cdot\sum_{i=0}^{n-2}(t^i\otimes
t^{n-2-i})\cdot\bigl(\alpha_i\delta(u_1)^2+\beta_i\delta(u_2)^2\bigl)
\end{equation}
in $A_\xi\otimes A_\xi$, with elements $u_1,\,u_2\in A_\xi$ and
$\alpha_i,\,\beta_i\in A_\xi\otimes A_\xi$, such that $\ol\alpha_i$,
$\ol\beta_i>0$ in $\R$ for all~$i$.
\end{thm}

Here, if $\alpha\in A_\xi\otimes A_\xi$, we denote by
$\ol\alpha\in\R$ the evaluation of $\alpha$ at
$(\xi,\xi)\in(X\times X)(\R)$. So $\ol\alpha$ is the image of
$\alpha$ under $A\otimes A\labelto{\text{mult}}A\labelto\xi\R$. The
essential point in Theorem \ref{txyidentaet} is that an identity
\eqref{txydecomp} can be chosen such that the residues $\ol\alpha_i$,
$\ol\beta_i$ are all strictly positive.

It is worthwile to isolate the generic situation $(m,n)=(1,2)$:

\begin{cor}\label{txygenericase}%
In Theorem \ref{txyidentaet} assume that $\eta=\pi(\xi)$ is a
nonsingular point of $C$ with simple tangent. Then there is an
identity
$$\pm T\>=\>\alpha\cdot\delta(u_1)^2+\beta\cdot\delta(u_2)^2$$
in $A_\xi\otimes A_\xi$ with $\ol\alpha$, $\ol\beta>0$ in $\R$.
\qed
\end{cor}

\begin{lab}
If Theorem \ref{txyidentaet} has been proved for one choice of
uniformizers $s,\,t$ at~$\xi$, then it holds for any choice.
We'll prove the identity for $s=t$ with $t$ chosen according to the
next lemma. This lemma allows us to assume that $A$ is generated by
two elements as an $\R$-algebra.
\end{lab}

\begin{lem}\label{i2gen}%
Let $X$ be a nonsingular affine curve over $\R$. Given any point
$\xi\in X(\R)$, there is an open affine neighborhood $U$ of $\xi$ on
$X$ such that there are $t,\,u,\,s\in\R[X]$ with $val_\xi(t)=1$ and
$\R[U]=\R[t,u]_s$.
\end{lem}

\begin{proof}
Choose an open neighborhood $V$ of $\xi$ on $X$ and a morphism
$\pi\colon V\to\A^2$ which is birational onto $Y:=\ol{\pi(V)}$ such
that $\pi(\xi)=\omega$ is a nonsingular point of $Y$ (see e.g.\
\cite[Problem 7.21]{fu}).
Then $\R[Y]$ is generated over $\R$ by two elements $t,\,u$, and we
can assume that $t$ is a local parameter of $Y$ at~$\omega$. Since
suitable neighborhoods of $\xi$ (on $X$) and $\omega$ (on $Y$) are
isomorphic under $\pi$, we are done.
\end{proof}

\begin{lab}\label{assnowon}%
Assume from now on that $A=\R[X_0]=\R[t,u]_s$ with $s(\xi)\ne0$ and
$val_\xi(t)=1$ (we may do so by Lemma \ref{i2gen}). Clearly, we can
also assume $val_\xi(u)\ge2$. By changing $s$ we can assume in
addition that the $A$-module $\Omega=\Omega_{A/\R}$ is freely
generated by~$dt$, and that $t$ generates the maximal ideal $\m_\xi$
of~$A$.
Writing $(m,n):=(m_\xi,n_\xi)$ (so $1\le m<n$), we may assume
$val_\xi(x)=m$ and $val_\xi(y)=n$.
Having arranged matters in this way, we'll establish a decomposition
\eqref{txydecomp} for the tangent tensor
$$T\>:=\>T_t(x,y)\>=\>\frac{dx}{dt}\delta(y)-\frac{dy}{dt}\delta(x)$$
in $A\otimes A$, with $u_1=t$ and $u_2=u$.
\end{lab}

\begin{lem}\label{i2gen2}%
The ideal $I=\ker(A\otimes A\labelto{\rm{mult}}A)$ of $A\otimes A$
is generated by $\delta(t)$ and $\delta(u)$.
\end{lem}

\begin{proof}
As an ideal, $I$ is generated by all elements $\delta(f)$, $f\in A$,
since $\sum_ia_i\otimes b_i=\sum_ib_i\delta(a_i)$ if $a_i,\,b_i\in A$
with $\sum_ia_ib_i=0$.
For $a,\,b\in A$ one has $\delta(ab)=a\delta(b)+b\delta(a)+
\delta(a)\delta(b)$.
If $s\in A$ is a unit of $A$ then $\delta(\frac1s)=
-(\frac1s\otimes\frac1s)\delta(s)$.
Since $A$ is a localization of $\R[t,u]$, the lemma follows from
these remarks.
\end{proof}

Let $J$ denote the kernel of the ring homomorphism
$A\otimes A\labelto{\text{mult}}A\labelto\xi A/\m_\xi=\R$,
$\alpha\mapsto\ol\alpha$, and note that $I\subset J$. Recall
$m=val_\xi(x)$ and $n=val_\xi(y)$. For notational convenience we
abbreviate $t_1:=t\otimes1$ and $t_2:=1\otimes t\in A\otimes A$, so
$\delta(t)=t_1-t_2$. Since $\m_\xi$ is generated by $t$, the ideal
$J$ is generated by $t_1$ and~$t_2$.

\begin{lem}\label{keycong}%
Let $a=\ol{xt^{-m}}$ and $b=\ol{yt^{-n}}$. Then $0\ne a,\,b\in\R$ and
$$T\>=\>ab\,t_2^{m-1}\cdot\bigl(\delta(t)^2(S+w)+\delta(u)^2w'
\bigr)$$
with
\begin{equation}\label{keyterm}%
S\>:=\>\frac{m\delta(t^n)-nt_2^{n-m}\delta(t^m)}{\delta(t)^2}\>=\>
\sum_{j=2}^n\Bigl(m\choose nj-n\choose mj\Bigr)\cdot t_2^{n-j}
\delta(t)^{j-2}
\end{equation}
and suitable $w,\,w'\in J^{n-1}$.
\end{lem}

\begin{proof}
Let $\m=\m_\xi\subset A$, the maximal ideal corresponding to~$\xi$.
The local expansions of $x,\,y\in A$ with respect to the local
parameter $t$ are $x=at^m+\cdots$, $y=bt^n+\cdots$. So
\begin{equation}\label{dxdt}%
\frac{d^ix}{dt^i}\>\equiv\>a\choose mii!\,t^{m-i}\ (\text{mod }
\m^{m-i+1})\quad(1\le i\le m),
\end{equation}
\begin{equation}\label{dydt}%
\frac{d^jy}{dt^j}\>\equiv\>b\choose njj!\,t^{n-j}\ (\text{mod }
\m^{n-j+1})\quad(1\le j\le n).
\end{equation}
By \eqref{generictaylor} we have
\begin{equation}\label{deltax}%
\delta(x)\ \equiv\ \sum_{i=1}^m\frac1{i!}\,\frac{d^ix}{dt^i}\,
\delta(t)^i\quad(\text{mod }I^{m+1})
\end{equation}
and
\begin{equation}\label{deltay}%
\delta(y)\ \equiv\ \sum_{j=1}^n\frac1{j!}\,\frac{d^jy}{dt^j}\,
\delta(t)^j\quad(\text{mod }I^{n+1})
\end{equation}
in $A\otimes A$. Substituting these into $T$ and observing that the
terms linear in $\delta(t)$ cancel, this gives
$$T\>\equiv\>\frac{dx}{dt}\sum_{j=2}^n\frac1{j!}\frac{d^jy}{dt^j}
\delta(t)^j-\frac{dy}{dt}\sum_{i=2}^m\frac1{i!}\frac{d^ix}{dt^i}
\delta(t)^i$$
modulo $\frac{dx}{dt}I^{n+1}+\frac{dy}{dt}I^{m+1}\subset
t_2^{m-1}I^{n+1}+t_2^{n-1}I^{m+1}$. Further, using approximations
\eqref{dxdt} and \eqref{dydt}, we get (recall $n>m$)
\begin{align*}
T & \equiv\ ab\biggl(mt_2^{m-1}\sum_{j=2}^n\choose njt_2^{n-j}
  \delta(t)^j-nt_2^{n-1}\sum_{i=2}^m\choose mit_2^{m-i}\delta(t)^i
  \biggr)
   \\
& =\ ab\,t_2^{m-1}\,\delta(t)^2\biggl(m\sum_{j=2}^n\choose nj
  t_2^{n-j}\delta(t)^{j-2}-n\sum_{i=2}^n\choose mit_2^{n-i}
  \delta(t)^{i-2}\biggr) \\
& =\ ab\,t_2^{m-1}\,\delta(t)^2\cdot S
\end{align*}
modulo $t_2^{m-1}M'$ where
$$M'\>:=\>I^{n+1}+t_2^{n-m}I^{m+1}+\bigl\langle t_2^{n+1-j}
\delta(t)^j,\ j=2,\dots,n\bigr\rangle.$$
Recall $I=\idl{\delta(t),\delta(u)}$ (Lemma \ref{i2gen2}). This
implies $I^3=\idl{\delta(t)^2,\,\delta(u)^2}\cdot I\subset
\langle\delta(t)^2,\delta(u)^2\rangle J$, and therefore
\begin{equation}\label{i3trick}%
I^r\>\subset\>\idl{\delta(t)^2,\,\delta(u)^2}\cdot J^{r-2},\quad
r\ge3.
\end{equation}
Let $M:=\langle\delta(t)^2,\delta(u)^2\rangle J^{n-1}$. By the
previous remark, all summands of $M'$ are contained in $M$ except
$t_2^{n-m}I^{m+1}$ in case $m=1$. So the lemma is already proved if
$m>1$.
To deal with the case $m=1$, replace \eqref{deltax} by the finer
approximation
$$\delta(x)\>\equiv\>\frac{dx}{dt}\delta(t)+\frac12\frac{d^2x}{dt^2}
\delta(t)^2\quad({\rm mod}\ I^3)$$
which again holds by \eqref{generictaylor}. Proceeding otherwise as
before, we get $T\equiv abt_2^{m-1}\delta(t)^2\cdot S$
modulo~$M=t_2^{m-1}M$, since the additional term $\frac{dy}{dt}\cdot
\frac12\frac{d^2x}{dt^2}\delta(t)^2$ lies in $\delta(t)^2J^{n-1}
\subset M$ (note $n-1=m+n-2$).
This proves the lemma in all cases.
\end{proof}

Recall $\delta(t)=t_1-t_2$.

\begin{lem}\label{lemsurpris}%
Let $1\le m<n$, let $S$ be defined as in \eqref{keyterm}. Then $S$
is equal to
{\small
\begin{equation}\label{polynomid}%
(n-m)\sum_{i=0}^{m-2}(i+1)t_1^it_2^{n-i-2}+m(n-m)t_1^{m-1}t_2^{n-m-1}
+m\sum_{j=0}^{n-m-2}(j+1)t_1^{n-2-j}t_2^j.
\end{equation}}
\end{lem}

\begin{proof}
Let $S_1$ denote the expression \eqref{polynomid}. It suffices to
prove $\delta(t)^2S_1=m\delta(t^n)-nt_2^{n-m}\delta(t^m)$ (see
\eqref{keyterm}), which is the identity
$$(t_1-t_2)^2\,S_1\>=\>mt_1^n-nt_1^mt_2^{n-m}+(n-m)t_2^n$$
of binary forms. This can be checked coefficient-wise.
\end{proof}

\begin{lab}\label{endproof}%
We now complete the proof of Theorem \ref{txyidentaet}.
Note that \eqref{polynomid} is a linear combination of all the
products $t_1^it_2^j$ (where $i,j\ge0$ and $i+j=n-2$) with
\emph{strictly positive} (integer) coefficients.
So, by Lemmas \ref{keycong} and \ref{lemsurpris}, we can write
$T=abt_2^{m-1}T'$ with
$$T'\>=\>\delta(t)^2\sum_{i=0}^{n-2}c_it_1^it_2^{n-2-i}+
\delta(t)^2w+\delta(u)^2w'$$
where $0<c_i\in\R$ and $w,\,w'\in J^{n-1}$. Further, since
$val_\xi(u)\ge2$ (see \ref{assnowon}), we have $\frac{du}{dt}\in\m$,
so $\delta(u)\in t_2\delta(t)+I^2$ by \eqref{generictaylor}. This
gives $\delta(u)^2\in t_2^2\delta(t)^2+t_2I^3+I^4\subset
\delta(t)^2J^2+I^3J$, and hence
$$\delta(u)^2t_1^it_2^{n-2-i}\in J^{n-2}\cdot
\bigl(\delta(t)^2J^2+I^3J\bigr)=\delta(t)^2J^n+I^3J^{n-1}\subset
\langle\delta(t)^2,\,\delta(u)^2\rangle J^n$$
for every $0\le i\le n-2$ (use \eqref{i3trick} again). So we can as
well write
\begin{equation}\label{vorletzte}%
T'\>=\>\delta(t)^2\cdot\Bigl(w+\sum_{i=0}^{n-2}c_it_1^it_2^{n-2-i}
\Bigr)+\delta(u)^2\cdot\Bigl(w'+\sum_{i=0}^{n-2}t_1^it_2^{n-2-i}
\Bigr)
\end{equation}
with new elements $w,\,w'\in J^{n-1}$.

Now the essential point is, the products $t_1^it_2^{n-2-i}$
($0\le i\le n-2$) generate the ideal $J^{n-2}$ of $A\otimes A$.
So we can express $w$ resp.\ $w'$ as
$$w\>=\>\sum_{i=0}^{n-2}w_it_1^it_2^{n-2-i},\quad w'\>=\>
\sum_{i=0}^{n-2}w'_it_1^it_2^{n-2-i}$$
with suitable elements $w_i,\,w'_i\in J$ ($0\le i\le n-2$). Combining
these with \eqref{vorletzte} finally gives
\begin{equation}\label{final}%
T'\>=\>\delta(t)^2\sum_{i=0}^{n-2}(c_i+w_i)t_1^it_2^{n-i-2}+
\delta(u)^2\sum_{i=0}^{n-2}(1+w'_i)t_1^it_2^{n-i-2}
\end{equation}
which shows that $T$ has the form asserted in Theorem
\ref{txyidentaet}.
\qed
\end{lab}


\section{Proof of the main theorem}\label{mainpf}%

\begin{lab}\label{1streduct}%
Let $K\subset\R^2$ be a closed convex \sa\ set. Ultimately we want to
prove $\sxdeg(K)\le2$ by applying Theorems \ref{sxdegsosx} and
\ref{txyidentaet}. To do this we start by making a series of
reductions. We can assume that $K$ is not contained in a
line and does not contain a half-plane.
Then $K$ is the convex hull of its boundary $\partial K$
\cite[Theorem 18.4]{ro},
and $\partial K$ is a \sa\ set of dimension one. So it suffices to
prove $\sxdeg(\ol{\conv(S)})\le2$ for every closed \sa\ set
$S\subset\R^2$ with $\dim(S)=1$.
If $S$ is decomposed as a finite union $S=S_1\cup\cdots\cup S_r$ of
\sa\ sets $S_i$, then it is enough to show $\sxdeg(\ol{\conv(S_i)})
\le2$ for $i=1,\dots,r$.
Indeed, if $K_i=\ol{\conv(S_i)}$ has $\sxdeg(K_i)\le2$ for every~$i$,
the same is true for $K=\ol{\conv(K_1\cup\cdots\cup K_r)}$
(Proposition \ref{sxconvhull}(b) and Corollary \ref{sxdegkpk}),
and clearly $K=\ol{\conv(S)}$.
In this way we can reduce to the case where $C\subset\A^2$ is an
irreducible curve of degree~$>1$, and $S\subset C(\R)$ is a closed
subset homeomorphic either to a circle or to a closed interval in
the line.
Since the curve $C$ has only finitely many singular points or points
with a higher order tangent, we can in addition assume that $S$
contains no such point except possibly as a boundary point of~$S$.
We can also assume that any $f\in P_S$ vanishes in at most one point
of~$S$.
\end{lab}

\begin{lab}\label{erzpk}%
For $S$ as in \ref{1streduct} let $K=\ol{\conv(S)}$, and let
$P=P_K=\{f\in\R[x,y]\colon f\ge0$ on~$S$, $\deg(f)\le1\}$. Let $E$ be
the union of the extreme rays of the convex cone $P$, so $E$ consists
of all $f\in P$ for which $f=f_1+f_2$ and $f_1,\,f_2\in P$ implies
$f_1,\,f_2\in\R_\plus f$.
Then $E$ is a \sa\ subset of $P$ and $P=\cone(E)$, the conic hull
of~$E$, by the Krein-Milman theorem.

Let $f\in E$, and assume that $f$ is not constant.
Then $\inf f(S)=0$.
In addition, if there is $b\in S$ with $f(b)=0$, then $f$ is tangent
to the curve $C$ at~$b$, or else $b$ is a boundary point of~$S$.
If $f>0$ on $S$ then the line $f=0$ is an asymptote of $C$ at
infinity. Note that $C$ has only finitely many such asymptotes.
\end{lab}

\begin{lab}
For proving $\sxdeg(K)\le2$ it is enough to show
$\sosx f^\otimes(a)\le2$ for every real closed field $R\supset\R$,
every $f\in E_R\subset R[x,y]$ and every $a\in S_R\subset C(R)$
(Theorem \ref{sxdegsosx}). When $f$ or $a$ has coordinates in $\R$
this holds trivially, since then the tensor $f^\otimes(a)$ lies in
$R\otimes1$ resp.\ in $1\otimes R$.
Therefore we only need to consider the case where $f=\tau_b$ is an
equation of the tangent to $C$ at a point $b\in S_R$ which is not
$\R$-rational.
In particular, $b$ is a nonsingular $R$-point of~$C$.

Neither of the points $a,\,b\in S_R$ needs to have bounded
coordinates in general. But this can be rectified by making a
suitable projective coordinate change over~$\R$ (we consider
$\A^2\subset\P^2$ in the standard way).
So we can assume that $a,\,b$ have coordinates in~$B$, the canonical
valuation ring of~$R$ (see \ref{valringb}). Let
$\ol a,\,\ol b\in S\subset C(\R)$ be their specializations. By
scaling we can also assume that the coefficients of $f=\tau_b$ lie in
$B$, and not all lie in $\m_B$. Then $\tau_b(a)\in B$ and
$\tau_b(a)\ge0$. If $\ol{\tau_b(a)}>0$ then $\sosx\tau_b^\otimes(a)
=1$ by Proposition \ref{keyobs}. So we can assume $\ol{\tau_b(a)}=0$.
In this case, the reduced linear polynomial $\ol{\tau_b}\in\R[x,y]$
is nonnegative on $S$ and vanishes in both $\ol a$, $\ol b\in S$. By
our assumptions (see \ref{1streduct}) we therefore have
$\ol a=\ol b$.
\end{lab}

\begin{lab}
In summary we can assume that $a,\,b\in C(B)$ are not $\R$-rational
but have the same specialization $\ol a=\ol b=:\eta\in C(\R)$, and
that $f\in B[x,y]$ is the tangent to $C$ at the point $b$. Let
$\pi\colon X\to C$ be the normalization of $C$, write $A=\R[X]$, and
let $a',\,b'\in X(B)$ be the preimages of $a,\,b$ under~$\pi$.
We have $\ol{a'}=\ol{b'}$ in $X(\R)$. Indeed, this can only fail if
$\eta$ is a singular point of $C$. But if $\eta$ is singular, then
$S$ contains only one half-branch centered at~$\eta$, by the initial
assumptions \ref{1streduct}, and so we still have $\ol{a'}=\ol{b'}$.
Denote this point by~$\xi$, and write $A_\xi:=\scrO_{X,\xi}$ for the
local ring of $X$ at~$\xi$, as in \ref{setup3}. The evaluation
homomorphism $A\to B$, $p\mapsto p(a')$ at~$a'$ extends to a ring
homomorphism $A_\xi\to B$, since $\ol{p(a')}=p(\ol a')=p(\xi)$ for
every $p\in A$.
Similarly we have an evaluation homomorphism $A_\xi\to B$, $q\mapsto
q(b')$ at~$b'$. So there is a well-defined ring homomorphism
\begin{equation}\label{homab}%
\phi\colon A_\xi\otimes A_\xi\to B\otimes B,\quad
p\otimes q\mapsto p(a')\otimes q(b').
\end{equation}
Let $T=T_s(x,y)\in A_\xi\otimes A_\xi$ be the tangent tensor (for
some local parameter $s$ at~$\xi$), as in \ref{setup2}. According to
\ref{setup2},
$f^\otimes(a)$ is the image of $T$ under the homomorphism
\eqref{homab}, up to a scaling factor of the form $c\otimes1$.
Hence the decomposition of $\pm T$ established in Theorem
\ref{txyidentaet} induces a corresponding decomposition of the tensor
$f^\otimes(a)$ in $B\otimes B$, via the homomorphism \eqref{homab}.
\end{lab}

\begin{lab}\label{laststepofproof}%
Let $(m,n)=(m_\xi,n_\xi)$ as in \ref{setup3}, and assume first that
$(m,n)=(1,2)$. By Corollary \ref{txygenericase}, combined with
Proposition \ref{keyobs}, we have $\sosx(\pm\phi(T))\le2$ for one
choice of the sign $\pm$. Therefore $\sosx f^\otimes(a)\le2$, see
Remark \ref{sgntenseval}.

Now assume $(m,n)\ne(1,2)$. Then, by the assumptions in
\ref{1streduct}, $\eta$ is an endpoint of $S$. So both $a,\,b$ lie on
the same local real halfbranch of $C$ centered at~$\eta$. Therefore
we can assume in Theorem \ref{txyidentaet} that the local uniformizer
$t$ is positive in $a$ and $b$ (otherwise replace $t$ by $-t$).
Reading the right hand side of \eqref{txydecomp} in $B\otimes B$ via
the homomorphism \eqref{homab}, we see again that this element is a
sum of squares of binomial tensors. This completes the proof of
Theorem \ref{mainthm}.
\qed
\end{lab}

\begin{rem}
One may wonder whether Theorem \ref{mainthm} extends to convex \sa\
sets $K\subset\R^2$ that are not closed. It is known that any such
$K$ is a spectrahedral shadow \cite{sch:curv}. However, we were not
able to decide whether always $\sxdeg(K)\le2$ holds. Given closed
convex subsets $T\subset S$ of $\R^2$, the question is whether the
convex set $(T\looparrowleft S)$ (see \cite[Theorem 3.8]{nt} and
\cite{sch:curv}, proof of Theorem 6.8) has $\sxdeg\le2$.
From Netzer's argument in \cite{nt} (proof of Theorem 3.8), we
only seem to get the bound $\sxdeg(T\looparrowleft S)\le4$. So
$\sxdeg(K)\le4$ holds for every convex \sa\ set $K\subset\R^2$, but
it is not clear whether this bound is sharp.
\end{rem}


\section{Constructive aspects}\label{constrasp}%

The proof of Theorem \ref{mainthm} in Sections \ref{maintech} and
\ref{mainpf} is essentially constructive. That is,
given a closed convex \sa\ set $K\subset\R^2$, one can (in principle)
find an explicit second-order cone representation of~$K$. We first
illustrate this in a particularly easy situation. After this we'll
sketch the general procedure.

\begin{lab}\label{avex}%
Let $f(t)\in\R[t]$ be a polynomial that is strictly convex on a
neighborhood of~$0$, say $f''(t)>0$ for $|t|<1$. We show how to
find a second order cone representation of the epigraph of $f$
$$K_a\>=\>\{(x,y)\in\R^2\colon y\ge f(x),\ |x|\le a\},$$
for some real number $0<a\le1$. The cone
$P_{K_a}\subset\R+\R x+\R y$ of
linear polynomials nonnegative on $K_a$ is generated by the tangent
$$\tau_v\>=\>y-f(v)-(x-v)f'(v)\ \in\R[x,y]$$
at $(v,f(v))$ for $|v|\le a$, together with the vertical lines
$a\pm x$.
Let us make the procedure of Theorem \ref{txyidentaet} explicit for
this example, in a neighborhood of the origin. The curve $X$ figuring
in this theorem is the affine line, so $A=\R[t]$. For $x=t$ and
$y=f(t)\in\R[t]$, the tangent tensor (\ref{setup2}) in
$\R[t]\otimes\R[t]$ is
$$T\>=\>T(x,y)\>=\>\frac{dx}{dt}\delta(y)-\frac{dy}{dt}\delta(x)\>=\>
\delta(f(t))-f'(t)\delta(t).$$
To simplify notation, write $\R[t]\otimes\R[t]=\R[u,v]$ where
$u=t\otimes1$ and $v=1\otimes t$. Then $\delta(t)=u-v$ and
$\delta(f(t))=f(u)-f(v)$, so expanding the above expression gives
\begin{equation}\label{pullbacktang}%
T\>=\>f(u)-f(v)-(u-v)f'(v)\>=\>(u-v)^2\sum_{k\ge2}\frac1{k!}
f^{(k)}(v)\cdot(u-v)^{k-2},
\end{equation}
which is the second order remainder in the Taylor expansion of~$f(u)$
around $v$. If we read $u,\,v$ as elements of $R$, then
\eqref{pullbacktang} is the tensor evaluation
$\tau_v^\otimes(u,f(u))\in R\otimes R$ of $\tau_v$ at the point
$(u,f(u))\in R^2$.

To arrive at an explicit second order cone representation of $K_a$,
it is enough to find a polynomial decomposition
\begin{equation}\label{posdecompos}%
\sum_{k\ge2}\frac1{k!}f^{(k)}(v)\cdot(u-v)^{k-2}\>=\>
\sum_{i=1}^mp_i(u)q_i(v)
\end{equation}
in such a way that $p_i(0),\,q_i(0)>0$ for each~$i$. This is possible
since substitution $v=u$ on the left gives $\frac12f''(v)>0$, a
strictly positive value. The proof of Proposition \ref{keyobs} shows
a constructive way for finding such $p_i,\,q_i$.

Fix an identity \eqref{posdecompos}, and let $a>0$ be such that
$p_i(t),\,q_i(t)\ge0$ for $|t|\le a$.
Let $V$ be the affine curve with coordinate ring
$\R[V]=\R[t,z_0,z_1,\dots,z_m]/\fra$, where $\fra$ is the ideal
generated by $z_i^2-p_i(t)$ ($1\le i\le m$) and $z_0^2+t^2-a^2$.
In other words, $\R[V]$ is obtained by adjoining square roots of
$p_1(t),\dots,p_m(t)$ and $a^2-t^2$ to $\R[t]$. Let
$\phi\colon V\to\A^2$ be defined by the ring homomorphism
$\phi^*\colon\R[x,y]\to\R[V]$ with $\phi^*(x)=t$ and
$\phi^*(y)=f(t)$. The image $\phi(V(\R))\subset\R^2$ of the real
locus of $V$ is the graph of $f|_{[-a,a]}$.
We have $\phi^*(a\pm x)=a\pm t=\frac1{2a}(z_0^2+(a\pm t)^2)$ in
$\R[V]$,
and
\begin{equation}\label{tvsos}%
\phi^*(\tau_v)\>=\>(t-v)^2\cdot\sum_{i=1}^mq_i(v)z_i^2
\end{equation}
in $\R[V]$ by \eqref{pullbacktang} and \eqref{posdecompos}. Since
$q_i(v)\ge0$ for $|v|\le a$, the cone $\phi^*(P_{K_a})$ consists of
sums of squares in $\R[V]$.
More precisely, let $U_i=\spn(z_i,tz_i)$ ($i=1,\dots,m$) and
$U_0=\spn(1,t)$, $U'_0=\spn(z_0)$. These are linear subspaces of
$\R[V]$ of dimension $\le2$, and
$$\phi^*(P_{K_a})\>\subset\>\Sigma U_0^2+\Sigma U_0'^2+\Sigma U_1^2+
\cdots+\Sigma U_m^2.$$
Therefore, if $A,\,B,\,C\in\R$ then $Ax+By+C\in P_{K_a}$ if and only
if there is an identity
$$At+Bf(t)+C\>=\>c(a^2-t^2)+\sum_{i=0}^mg_i(t)\cdot p_i(t)$$
in $\R[t]$, with $p_0(t)=1$, $0\le c\in\R$ and nonnegative quadratic
polynomials $g_i(t)=a_it^2+2b_it+c_i$ (i.e.\ with
$\choose{a_i\ b_i}{b_i\ c_i}\succeq0$), $i=0,\dots,m$.
This is a semidefinite representation for the cone $P_{K_a}$ that
shows $\sxdeg(P_{K_a})=2$.
Dualizing this representation (c.f.\ Proposition \ref{sxdual2},
Corollary \ref{sxdegkpk}) we obtain a second-order cone
representation for~$K_a$.
\end{lab}

\begin{lab}\label{factoriz}%
From an identity \ref{posdecompos} with $p_i(0),\,q_i(0)>0$, one
immediately reads off an $(\sfS^2_\plus)^m$-factori\-zation of $K_a$,
see Remark \ref{gptrem}.
Indeed, if $|u|,\,|v|\le a$, and if $\tau_v\in P_{K_a}$ is the
positive tangent at $(v,f(v))$ as before, the matrices
$$A_i(u)\>:=\>p_i(u)\begin{pmatrix}1&u\\u&u^2\end{pmatrix},\quad
B_i(v)\>:=\>q_i(v)\begin{pmatrix}v^2&-v\\-v&1\end{pmatrix}$$
($1\le i\le m$) are psd of rank $\le1$ and satisfy
$$\tau_v(u,f(u))\>=\>\sum_{i=1}^m\bigl\langle A_i(u),\,B_i(v)
\bigr\rangle,$$
by \eqref{posdecompos}.

The question for the existence of such a factorization was raised by
Gennadiy Averkov (Oberwolfach, June 2019), in the case of the
polynomial $f(x)=x^2-x^6$ (which is strictly convex for
$|x|<\root^4\of{15}\approx0.5081$). One possible decomposition
\eqref{posdecompos} in this case is
$$1-u^4-2u^3v-3u^2v^2-4uv^3-5v^4\>=\>2p_3q_1+3p_2q_2+4p_1q_3+p_4+
q_4$$
with $p_i(u)=a^i+u^i$, $q_j(v)=a^j-v^j$ ($i,j=1,2,3$),
$p_4(u)=\frac12-4a^4-4a^3u-3a^2u^2-2au^3-u^4$ and
$q_4(v)=\frac12-5a^4+2a^3v+3a^2v^2+4av^3-5v^4$. In this specific
decomposition we have $p_i(u)\ge0$, $q_j(v)\ge0$ for $|u|,\,|v|\le a$
and $i,j=1,\dots,4$, as long as $a\le1/\root^4\of{28}\approx0.4347$.
\end{lab}

\begin{lab}
The discussion in \ref{avex} and \ref{factoriz} was particularly easy
for several reasons. One is that the Zariski closure of the boundary
of $K_a$ is a nonsingular rational curve. Moreover, all tangents are
ordinary, i.e.\ they have contact order two. Existence
of higher order tangents forces a refined decomposition of the
tangent tensor, see Corollary \ref{txygenericase} versus Theorem
\ref{txyidentaet}. Nonrational boundary curves, or singularities of
the boundary, will further complicate the picture, as detailed in
Section~4.

To illustrate the influence of higher order tangents, consider the
problem of representing the epigraph of a polynomial $f(x)$ that is
only nonstrictly (instead of strictly) convex. For example, if
$f=c_mx^m+c_{m+1}x^{m+1}+\cdots$ with $c_m>0$ and $m\ge2$, the
tangent tensor \eqref{pullbacktang} becomes
$$T\>=\>(u-v)^2\cdot\Bigl(c_m\sum_{i=0}^{m-2}(i+1)u^{m-2-i}v^i+
(\text{higher order terms})\Bigr),$$
and one needs to find polynomials $p_{ij},\,q_{ij}$ with
$$\sum_{k\ge2}\frac1{k!}f^{(k)}(v)\cdot(u-v)^{k-2}\>=\>
\sum_{i=0}^{m-2}u^{m-2-i}v^i\sum_jp_{ij}(u)q_{ij}(v)$$
such that $p_{ij}(0)>0$ and $q_{ij}(0)>0$ (analogue of
\eqref{posdecompos}, see the proof of Theorem \ref{txyidentaet}).
\end{lab}

\begin{lab}
Suppose we want to find an explicit second-order cone representation
for an arbitrary given closed convex \sa\ set $K\subset\R^2$. By
decomposing into finitely many pieces, we can assume that $K$ is the
closed convex hull of a \sa\ set $S\subset C(\R)$ as in
\ref{1streduct}, where $C\subset\A^2$ is an irreducible curve.
Let $\pi\colon X\to C$ be the normalization, and let $\xi\in X(\R)$
with $\eta=\pi(\xi)\in S$. Write $A=\R[X]$. Since the proof of
Theorem \ref{txyidentaet} was constructive, we can find a
decomposition \eqref{txydecomp} of the tensor
$T(x,y)\in A_\xi\otimes A_\xi$ as in
Theorem \ref{txyidentaet}, with explicit elements
$\alpha_i,\,\beta_i\in A_\xi\otimes A_\xi$ and $u_1,\,u_2\in A_\xi$.
Each of the $\alpha_i,\,\beta_i$ can be written (explicitly) as a sum
of tensors $a_\nu\otimes b_\nu$ with $a_\nu,\,b_\nu\in A_\xi$ and
$a_\nu(\xi),\,b_\nu(\xi)>0$, see Proposition \ref{keyobs} and its
proof. Let $S'\subset S$ be a closed neighborhood of $\eta$ inside
$S$ on which all the $a_\nu$ and the $b_\nu$ are strictly positive.
Extend the ring $A$ by adjoining square roots of all the (finitely
many) elements $a_\nu$, let $\psi\colon V\to X$ be the morphism so
defined, and let $\phi=\pi\comp\psi\colon V\to C$. Similar to the
arguments in \ref{avex}, we see that we obtain an explicit
second-order cone representation for the closed convex hull of~$S'$.

Working locally around every point $\eta$ of $S$ in this way, the set
$S$ is covered by finitely many local patches. Patching together
these local representations \`a~la Proposition \ref{sxconvhull}, one
can then arrive at a global representation for~$K$.
\end{lab}


\end{document}